\documentclass[reqno,a4paper,twoside,draft]{amsart}
\usepackage{enumitem}

\setenumerate{label=\textnormal{(\arabic*)}}

\usepackage{amsmath,amssymb,dsfont,verbatim,bm,mathtools,geometry,fge,bm}
\usepackage[raggedright]{titlesec}

\usepackage{quoting}
\usepackage[latin1]{inputenc}
\usepackage[raggedright]{titlesec}
\usepackage{mathtools}
\usepackage{tikz}
\usepackage{float,subfig}
\usepackage{amsrefs}
\usepackage[linesnumbered,ruled,vlined]{algorithm2e}
\usepackage{booktabs}

\titleformat{\chapter}[display]
{\normalfont\huge\bfseries}{\chaptertitlename\\thechapter}{20pt}{\Huge}
\titleformat{\section}
{\normalfont\Large\bfseries\center}{\thesection}{1em}{}
\titleformat{\subsection}
{\normalfont\large\bfseries}{\thesubsection}{1em}{}
\titleformat{\subsubsection}[runin]
{\normalfont\normalsize\bfseries}{\thesubsubsection}{1em}{}
\titleformat{\paragraph}[runin]
{\normalfont\normalsize\bfseries}{\theparagraph}{1em}{}
\titleformat{\subparagraph}[runin]
{\normalfont\normalsize\bfseries}{\thesubparagraph}{1em}{}
\titlespacing*{\chapter} {0pt}{50pt}{40pt}
\titlespacing*{\section} {0pt}{3.5ex plus 1ex minus .2ex}{2.3ex plus .2ex}
\titlespacing*{\subsection} {0pt}{3.25ex plus 1ex minus .2ex}{1.5ex plus .2ex}
\titlespacing*{\subsubsection}{0pt}{3.25ex plus 1ex minus .2ex}{1.5ex plus .2ex}
\titlespacing*{\paragraph} {0pt}{3.25ex plus 1ex minus .2ex}{1em}
\titlespacing*{\subparagraph} {\parindent}{3.25ex plus 1ex minus .2ex}{1em}
\usepackage{titletoc}
\contentsmargin{2.5em}
\dottedcontents{section}[1.5em]{}{2.8em}{1pc}
\dottedcontents{subsection}[3.7em]{}{3.4em}{1pc}
\dottedcontents{subsubsection}[7.6em]{}{3.2em}{1pc}
\dottedcontents{paragraph}[10.3em]{}{3.2em}{1pc}

\newtheorem{theorem}{Theorem}[section]
\newtheorem{lemma}[theorem]{Lemma}
\newtheorem{proposition}[theorem]{Proposition}
\newtheorem{corollary}[theorem]{Corollary}

\theoremstyle{definition}
\newtheorem{definition}[theorem]{Definition}

\newtheorem{notation}[theorem]{Notation}
\newtheorem{example}[theorem]{Example}
\theoremstyle{remark}
\newtheorem{remark}[theorem]{Remark}

\DeclareMathOperator{\Aut}{Aut}

\DeclareMathOperator{\Coeff}{Coeff}

\DeclareMathOperator{\ide}{id}
\DeclareMathOperator{\ev}{ev}

\DeclareMathOperator{\Supp}{Supp}

\DeclareMathOperator{\en}{en}

\DeclareMathOperator{\st}{st}

\DeclareMathOperator{\Res}{Res}

\DeclareMathOperator{\Succ}{Succ}
\DeclareMathOperator{\Pred}{Pred}

\DeclareMathOperator{\Dir}{Dir}
\DeclareMathOperator{\dir}{dir}

\begin{document}

\title{The Jacobian Conjecture: Approximate roots and intersection numbers}

\author{Jorge A. Guccione}
\address{Departamento de Matem\'atica\\ Facultad de Ciencias Exactas y Naturales-UBA, Pabell\'on~1-Ciudad Universitaria\\ Intendente Guiraldes 2160 (C1428EGA) Buenos Aires, Argentina.}
\address{Instituto de Investigaciones Matem\'aticas ``Luis A. Santal\'o"\\ Facultad de Ciencias Exactas y Natu\-rales-UBA, Pabell\'on~1-Ciudad Universitaria\\ Intendente Guiraldes 2160 (C1428EGA) Buenos Aires, Argentina.}
\email{vander@dm.uba.ar}

\author{Juan J. Guccione}
\address{Departamento de Matem\'atica\\ Facultad de Ciencias Exactas y Naturales-UBA\\ Pabell\'on~1-Ciudad Universitaria\\ Intendente Guiraldes 2160 (C1428EGA) Buenos Aires, Argentina.}
\address{Instituto Argentino de Matem\'atica-CONICET\\ Saavedra 15 3er piso\\ (C1083ACA) Buenos Aires, Ar\-gentina.}
\email{jjgucci@dm.uba.ar}

\author{Rodrigo Horruitiner}
\address{Pontificia Universidad Cat\'olica del Per\'u, Secci\'on Matem\'aticas, PUCP, Av. Universitaria 1801, San Miguel, Lima 32, Per\'u.}
\email{rhorruitiner@pucp.edu.pe}

\author{Christian Valqui}
\address{Pontificia Universidad Cat\'olica del Per\'u, Secci\'on Matem\'aticas, PUCP, Av. Universitaria 1801, San Miguel, Lima 32, Per\'u.}
\address{Instituto de Matem\'atica y Ciencias Afines (IMCA) Calle Los Bi\'ologos 245. Urb San C\'esar.
La Molina, Lima 12, Per\'u.}
\email{cvalqui@pucp.edu.pe}

\begin{abstract}
We translate the results of Yansong Xu into the language of~\cite{GGV1}, obtaining nearly the same 
formulas for the intersection number of Jacobian pairs, but with an inequality instead of an equality.
\end{abstract}

\maketitle

\setcounter{tocdepth}{2}
\tableofcontents

\section*{Introduction} The Jacobian Conjecture (JC) in dimension two stated by Keller in \cite{K} says that any pair of polynomials $P,Q\in
L\coloneqq K[x,y]$, with $[P,Q]\coloneqq \partial_x P \partial_y Q - \partial_x Q \partial_y P\in K^{\times}$, defines an invertible automorphism of
$L$. If this conjecture is false, then we can find a counterexample such that the shape of the support of the components $P\coloneqq f(x)$,
$Q\coloneqq f(y)$, is contained in rectangles $(0,0)$, $m(a,0)$, $m(a,b)$, $m(0,b)$ and $(0,0)$, $n(a,0)$, $n(a,b)$, $n(0,b)$, such that $m(a,b)$ is
in the support of $P$ and $n(a,b)$ is in the support of $Q$.  In a recent
paper~\cite{X} Yangsong Xu gives two formulas for the intersection number of possible counterexamples, which we call $I_M$ and
$I_m$. If the formulas were true, we would be able to discard many infinite families of possible counterexamples to the Jacobian conjecture
described in~\cite{GGHV}.

When we translated the result and the proofs of~\cite{X} into  the language of~\cite{GGV1}, we obtained the same formula for $I_M$, but for 
$I_m$ we obtained only an inequality, consequently we cannot discard the infinite families as desired.

Hence, the main result of the present article is the translation of the concept of approximate roots into our language (see~\cite{GGV1},
also~\cite{GGV2} and~\cite{GGHV}), which requires a dictionary from Moh's language to our language. It is interesting on its own, and the modified formulas could
help understand some features of Moh's methods.
 
Along this paper we will freely use the notations of \cite{GGV1}.

\section{General lower side corners}
\label{lower side corners}
Let $l\in \mathds{N}$ and let $(P,Q)\in L^{(l)}$ be an $(m,n)$-pair (see~\cite{GGV1}*{Definition~4.3}). In this section we take
$(\rho,\sigma)\in  ](0,-1),(1,1)]$ such that
\begin{equation*}\label{esquina inferior}
\frac 1m \en_{\rho,\sigma}(P)=\frac 1n \en_{\rho,\sigma}(Q)\eqqcolon (a/l,b)\quad\text{and}\quad a/l>b>0
\end{equation*}
(assuming that such a direction exists).
Note that $\rho>0$.  Assume that $u_p\coloneqq v_{\rho,\sigma}(P)>0$. Then the points $(a/l,b)$ and $(c/l,d)\coloneqq \frac 1m \st_{\rho,\sigma}(P)$ must satisfy certain conditions. Our purpose in this section is to analyse them.

\begin{proposition}\label{jacobiano se anula}
Under the above assumptions, $[\ell_{\rho,\sigma}(P),\ell_{\rho,\sigma}(Q)]=0$.
\end{proposition}
\begin{proof}
By~\cite{GGV1}*{Proposition~1.13} it suffices to prove that $v_{\rho,\sigma}(P)+v_{\rho,\sigma}(Q)>\rho+\sigma$. If $\rho+\sigma\le 0$, then this is true, since $v_{\rho,\sigma}(Q)=\frac nm v_{\rho,\sigma}(P)>0$; while if $\rho+\sigma>0$, then since $\frac al >b\ge 1$ and $\rho>0$, we have
$$
v_{\rho,\sigma}(P)+v_{\rho,\sigma}(Q)=(m+n)\left(\rho \frac al+\sigma b\right)>(m+n)b(\rho+\sigma)>\rho+\sigma,
$$
as desired.
\end{proof}

\begin{proposition}\label{direcciones positivas}
Under the above assumptions, if $\rho+\sigma>0$, then $\rho|l$ and there exist $\lambda,\mu\in K^{\times}$, such that $\ell_{\rho,\sigma}(P)=\lambda x^{u_p/\rho}(z-\mu)^{mb}$,  where $z\coloneqq x^{-\sigma/\rho}y$.
\end{proposition}

\begin{proof}
By~\cite{GGV1}*{Theorem~2.6} there exists a $(\rho,\sigma)$-homogeneous element $F\in L^{(l)}$ such that
\begin{itemize}

\item[-] $v_{\rho,\sigma}(F)=\rho+\sigma$,

\smallskip

\item[-] $[F,\ell_{\rho,\sigma}(P)]= \ell_{\rho,\sigma}(P)$,

\smallskip

\item[-] $\st_{\rho,\sigma}(P)\sim\st_{\rho,\sigma}(F)$ or $\st_{\rho,\sigma}(F) = (1,1)$,

\smallskip

\item[-] $\en_{\rho,\sigma}(P)\sim\en_{\rho,\sigma}(F)$ or $\en_{\rho,\sigma}(F)= (1,1)$.

\smallskip

\end{itemize}
If $\en_{\rho,\sigma}(P)=m(a/l,b)\sim \en_{\rho,\sigma}(F)$, then there exists $\lambda>0$ such that $\en_{\rho,\sigma}(F) =\lambda (a/l,b)$. So
$$
\rho + \sigma = v_{\rho,\sigma}(F) = \rho \lambda \frac al + \lambda\sigma b > \lambda b(\rho+\sigma) \Longrightarrow 0<\lambda b < 1,
$$
which is impossible, since $\lambda b = v_{0,1}(\en_{\rho,\sigma}(F))\in \mathds{Z}$. Consequently $\en_{\rho,\sigma}(F) = (1,1)$, and hence $\st_{\rho,\sigma}(F)=(1+\sigma/\rho,0)$, by~\cite{GGV1}*{Proposition~2.11(2)}. Thus $\rho|l$ and $\st_{\rho,\sigma}(P)\sim\st_{\rho,\sigma}(F)$, which implies $v_{0,1}(\st_{\rho,\sigma}(P))=0$. Write
$$
F=x^{\frac{u}{l}}y^v f(z)\quad\text{and}\quad \ell_{\rho,\sigma}(P)={x^{\frac{c}{l}}y^d} p(z) \quad\text{with $p(0)\ne 0\ne f(0)$.}
$$
Note that $v={d}=0$, $\rho c/l = u_p$, $v_{0,1}(\en_{\rho,\sigma}(P)) = mb$ and $f(z)=\lambda_1 (z-\mu)$ for some $\lambda_1,\mu\in K^{\times}$. By~\cite{GGV1}*{Proposition~2.11(1)} we have $\ell_{\rho,\sigma}(P)=\lambda x^{u_p/\rho} (z-\mu)^{mb}$, for some $\lambda\in K^{\times}$, which concludes the proof.
\end{proof}

By \cite{GGV1}*{Proposition~2.1(2)} (which applies thanks to Proposition~\ref{jacobiano se anula}) we know that there exist $\lambda_P,\lambda_Q\!\in\! K^{\times}$ and a $(\rho,\sigma)$-ho\-mo\-ge\-neous element $R\in L^{(l)}$, such that
$$
\ell_{\rho,\sigma}(P)=\lambda_P R^m\qquad\text{and}\qquad \ell_{\rho,\sigma}(Q) = \lambda_Q R^n.
$$
Let $\lambda\in K^{\times}$ and $R_0\in L^{(l)}$ be a $(\rho,\sigma)$-homogeneous element such that $\ell_{\rho,\sigma}(P) = \lambda R_0^h$ with $h$ maximum (consequently $m\mid h$ and we can assume that $R=R_0^{h/m}$ and $\lambda_P = \lambda$). Arguing as in \cite{GGV2}*{Corollary~2.6} we obtain that there exist $i\ge 0$ and a $(\rho,\sigma)$-homogeneous element $G\in L^{(l)}$ such that $[G,R] = R^i$.

\smallskip

Let $(a/l,b),(c/l,d)\in\frac 1l\mathds{Z}\times \mathds{Z}$ {such that} $a/l>b>d\ge 0$ and $a>c>0$. {Assume also that} $b-d<a/l-c/l$ (we do not assume the existence of $P$ and $Q$ at this point). It is well known that for each $(r/l,s)\in \frac{1}{l} \mathds{Z} \times\mathds{Z}\setminus\mathds{Z}(1,1)$ there exists a unique $(\varrho,\varsigma)\in\mathfrak{V}_{>0}$, which we denote by $\dir(r/l,s)$, such that $v_{\varrho,\varsigma}(r/l,s)=0$. Set $(\rho,\sigma)\coloneqq -\dir((a/l,b)-(c/l,d))$ and note that $(0,-1)<(\rho,\sigma)<(1,-1)$. We will analyse the existence of $i\in\mathds{N}$ and $(\rho,\sigma)$-ho\-mo\-ge\-neous elements $R,G\in L^{(l)}$, such that
\begin{equation}\label{determinacion de abcd}
v_{\rho,\sigma}(R)>0,\quad [G,R]=R^i,\quad (a/l,b)=\en_{\rho,\sigma}(R)\quad\text{and}\quad (c/l,d)=\st_{\rho,\sigma}(R).
\end{equation}
Let $\ell\in \mathds{N}$ be minimal with $\ell v_{\rho,\sigma}(R) +\rho+\sigma>0$. By \cite{GGV2}*{{Proposition~3.12}}, we know that if there exist {$i\in \mathds{N}$ and $R,G\in L^{(l)}$ satisfying~\eqref{determinacion de abcd}, and such that}
\begin{equation}\label{no potencia de un lineal}
R\ne \lambda x^{\frac{u}{\rho}}h^j(z)\quad\text{for all $\lambda\in K^{\times}$, $j\in\mathds{N}$, $z\coloneqq x^{-\frac{\sigma}{\rho}}y$ and all linear polynomial $h$},
\end{equation}
then {there exists $\vartheta,t'\in \mathds{N}$ such that}
\begin{equation}\label{condicion}
{\vartheta\le N_1,\qquad t'<\ell\vartheta \qquad\text{and}\qquad  (\rho,\sigma) = -\dir \Bigl(t'\Bigl(\frac{c}{l},d\Bigr) + \vartheta(1,1)\Bigr),}
\end{equation}
{where $N_1\coloneqq \gcd(a-c,b-d)$, or}
\begin{equation}\label{condicion'}
{d>0,\qquad \vartheta\mid N_2,\qquad t'<\ell\vartheta \qquad\text{and}\qquad (\rho,\sigma) = -\dir \Bigl(t'\Bigl(\frac{c}{l},d\Bigr) + \vartheta(1,1)\Bigr),}
\end{equation}
where $N_2\coloneqq \gcd(c,d)$. By~\cite{GGV2}*{Remark~3.13}
$$
\frac{\vartheta}{t'} = -\frac{\rho a/l+\sigma b}{\rho+\sigma}.
$$
Hence
$$
s\coloneqq \frac{\rho a+\sigma bl}{\gcd(\rho l+\sigma l,\rho a+\sigma bl )}\biggl| \vartheta,
$$
and so we can take (and we do it) $\vartheta=s$ in~\eqref{condicion} and~\eqref{condicion'}.

\smallskip

We suspect that the existence of $\vartheta$ and $t'$ satisfying the conditions in~\eqref{condicion} or in~\eqref{condicion'} is sufficient for the existence of $i\in \mathds{N}$ and two $(\rho,\sigma)$-homogeneous elements $R,G\in L^{(l)}$, such that the conditions in~\eqref{determinacion de abcd} and~\eqref{no potencia de un lineal} are satisfied (with $(c/l,d)\coloneqq \st_{\rho,\sigma}(R)$), but at the moment we have no proof.

\begin{remark}\label{s igual a b}
Since $N_2<b$, if $s=b$, then necessarily $b\le N_1$. So, by~\cite{GGV2}*{Proposition~3.12(2)} there exists a linear factor with multiplicity $b$, which contradicts~\eqref{no potencia de un lineal}. Consequently $s<b$.
\end{remark}

\begin{remark}
By~\cite{GGV2}*{Theorem~3.4} in~\eqref{determinacion de abcd} we can assume that $i$ is minimum such that
$$
v_{\rho,\sigma}(R)(i-1) + \rho + \sigma \ge 0,
$$
or, equivalently, that $i= \left\lceil 1- \frac{\rho+\sigma}{v_{\rho,\sigma}(R)} \right\rceil$.
\end{remark}

In the case $b=2$ we can establish necessary and sufficient conditions on $a$, $l$ for the existence of $c\in \mathds{N}$, $d\in\{0,1\}$ and two $(\rho,\sigma)$-homogeneous elements $R,G\in L^{(l)}$ satisfying the conditions of~\eqref{determinacion de abcd}, if we assume that $R$ satisfies~\eqref{no potencia de un lineal}. This additional condition corresponds to the existence of split roots (see Definition~\ref{pi-root}). Before we establish the result we note that
$$
(0,-1)<(\rho,\sigma)<(1,-1)\qquad\text{and}\qquad (\rho,\sigma) = -\dir \Bigl(\frac{a-c}{l},b-d\Bigr) \sim (lb-ld,c-a)
$$
implies $c<a$ and $b-d<a/l-c/l$.

\begin{proposition}\label{b igual a 2} Let $a,l\in \mathds{N}$ be such that $a/l>2$ and set $b\coloneqq 2$. Let $(\rho,\sigma)\in ](0,-1),(1,-1)[$ be a direction, and let
$$
\vartheta\coloneqq \frac{\rho a+\sigma bl}{\gcd(\rho l+\sigma l,\rho a+\sigma bl)}.
$$
The following assertions are equivalent:

\begin{enumerate}

\smallskip

\item {There exist $c\in\mathds{N}$, $d\in\{0,1\}$ and two $(\rho,\sigma)$-homogeneous elements $R,G\in L^{(l)}$ satisfying the conditions in~\eqref{determinacion de abcd} and~\eqref{no potencia de un lineal}.}

\smallskip

\item {There exist $c\in\mathds{N}$ and two $(\rho,\sigma)$-homogeneous elements $R,G\in L^{(l)}$ satisfying the conditions in~\eqref{determinacion de abcd} and~\eqref{no potencia de un lineal} with $d=1$.}

\smallskip

\item {$\vartheta = 1$, $v_{\rho,\sigma}(a/l,2)>0$ and there exist $c\in\mathds{N}$ such that
\begin{equation}\label{rho sigma}
(\rho,\sigma) = -\dir \Bigl(\frac{a-c}{l},1\Bigr) = -\dir \Bigl(t'\Bigl(\frac{c}{l},1\Bigr) + (1,1)\Bigr),
\end{equation}
for some $0<t'<\ell$, where $\ell\in \mathds{N}$ is minimal with $\ell v_{\rho,\sigma}(a/l,2)+\rho+\sigma>0$.}

\smallskip

\item {There exists $\Delta\in\mathds{N}$ with $l<\Delta<a/2$ such that $a-2\Delta \mid \Delta-l$.}

\end{enumerate}
{Moreover, $(\rho,\sigma)\sim (l,-\Delta)$.}

\end{proposition}

\begin{proof} 1) $\Rightarrow$ 2)\enspace {Suppose $d=0$ and write
$$
R=\lambda x^{\frac{u}{l}}(z-\alpha_1)(z-\alpha_2)\quad\text{with $z\coloneqq x^{-\frac{\sigma}{\rho}}y$.}
$$
Note that by~\eqref{no potencia de un lineal} we have $\alpha_1\ne \alpha_2$. Note also that $\rho u/l = 2\sigma + \rho a/l$, and hence $u = (2l\sigma+\rho a)/\rho$. Moreover since $b-d = 2$,
$$
(2l,c-a)\cdot \left(\frac{a}{l}-\frac{c}{l},b-d\right)= 2(a-c) - (c-a) (b-d) = 0,
$$
and consequently $(\rho,\sigma)\sim (2l,c-a)$. Finally, since $d=0$ necessarily~\eqref{condicion} is satisfied. We claim that $2|a-c $. In fact,
$$
0 = (2l,c-a)\cdot \Bigl(t'\Bigl(\frac{c}{l},0\Bigr) + \vartheta(1,1)\Bigr) = 2ct' + (c-a)\vartheta,
$$
which implies $2|a-c$, because otherwise $2\mid \vartheta \le N_1 = \gcd(a-c,2)= 1$. Set $\Delta\coloneqq (a-c)/2$ and consider the automorphism $\varphi$ of $L^{(l)}$ defined by $\varphi(x^{1/l})\coloneqq x^{1/l}$ and $\varphi(y)\coloneqq y+\alpha_1 x^{-\Delta/l}$. Using that $(\rho,\sigma)\sim (l,-\Delta)$ it is easy to check that
$$
\varphi(R)=\lambda x^{\frac{u}{l}}z(z-(\alpha_2-\alpha_1)).
$$
By~\cite{GGV1}*{Proposition~3.10}, we know that $[\varphi(G),\varphi(R)]=\varphi(R)^i$ and an easy computation shows that $\en_{\rho,\sigma}(\varphi(R)) =(a/l,b)$ and $\st_{\rho,\sigma}(\varphi(R)) =((a-\Delta)/l,1)$. So, replacing $R$ by $\varphi(R)$ yields $d=1$.}

\smallskip

\noindent 2) $\Rightarrow$ 1)\enspace {This is trivial.}

\smallskip

\noindent 2) $\Rightarrow$ 3)\enspace {Since $d=1$, we have $N_1=N_2=1$. Hence, $\vartheta = 1$ and equality~\eqref{rho sigma} is satisfied for some $0<t'<\ell$. Moreover it is clear that}
$$
{v_{\rho,\sigma}\Bigl(\frac{a}{l},2\Bigr)=v_{\rho,\sigma}(R)>0 \qquad\text{and}\qquad  (\rho,\sigma) = -\dir \bigl(\en_{\rho,\sigma}(R)-\en_{\rho,\sigma}(R)\bigr) = -\dir \Bigl(\frac{a-c}{l},1\Bigr).}
$$
\smallskip

\noindent 3) $\Rightarrow$ 4)\enspace {Since
$$
(l,c-a)\cdot \left(\frac{a}{l}-\frac{c}{l},1\right) = 0,
$$
we have $(\rho,\sigma)\sim (l,-\Delta)$, where $\Delta\coloneqq a-c$. Thus, by~\eqref{rho sigma},
$$
0 = (l,-\Delta)\cdot \Bigl(t'\Bigl(\frac{a-\Delta}{l},1\Bigr) + (1,1)\Bigr) = t'a-2t'\Delta + l -\Delta,
$$
which implies that $a-2\Delta|l-\Delta$, as desired. Since $(\rho,\sigma)\sim (l,-\Delta)$ and $v_{\rho,\sigma}(a/l,2)>0$, we have
$$
a-2\Delta= (l,-\Delta)\cdot \left(\frac al,2\right) = \frac{l}{\rho}(\rho,\sigma)\cdot \left(\frac al,2\right) > 0,
$$
and so $\Delta<a/2$. Finally, the computation $l-\Delta=\frac {l}{\rho}(\rho+\sigma)<0$ shows that $\Delta>l$.}

\smallskip

\noindent 4) $\Rightarrow$ 2)\enspace {Set $c\coloneqq a-\Delta$, $z\coloneqq x^{\Delta/l}y$ and $(\rho,\sigma)\coloneqq -\dir((a/l,2)-(c/l,1))$. Since $0<l<\Delta$, the inequalities $(0,-1)<(\rho,\sigma)<(1,-1)$ hold. Let $k_1\in \mathds{N}$ be such that $k_1(a-2\Delta)=\Delta-l$ and let $g(z)$ be a polynomial such that $g'(z)=z^{k_1}(1+z)^{k_1}$. A straightforward computation shows that
$$
R\coloneqq x^{\frac{a-2\Delta}{l}}z(1+z)=x^{\frac{c}{l}}y(1+z)\qquad\text{and}\qquad G\coloneqq \frac{l}{2\Delta - a}g(z),
$$
satisfy
$$
\left(\frac al,2\right)=\en_{\rho,\sigma}(R),\quad \left(\frac cl,1\right)=\st_{\rho,\sigma}(R),\quad v_{\rho,\sigma}(R)>0\quad\text{and}\quad
[G,R]=R^{k_1+1},
$$
as we want.}
\end{proof}

\section{Two formulas for the Intersection number}

Recall that the intersection number of two bivariate polynomials {$P$ and $Q$} is defined by $I(P,Q)\coloneqq
\deg_x({\Res_y}(P,Q))$, where ${\Res_y(P,Q)}$ denotes the resultant of $P$ {and} $Q$ as polynomials in
$y$. In~\cite{X}, the author defines for a Jacobian pair $(P,Q)$ the polynomial $P_\xi\coloneqq P(x,y)-\xi${,} where $\xi$ is a generic
element of the field $K$, and gives two different formulas for $I(P_\xi,Q)$, one in terms of the major roots in~\cite{X}*{Theorem~5.1} and the other
in terms of the minor roots in~\cite{X}*{Theorem~4.7}. We will prove the first formula using our language in Theorem~\ref{Interseccion con raices mayores}, and 
instead of the equality in the formula for $I_m$ we will prove an inequality in Theorem~\ref{numero de interseccion con raices menores}. In order to do this, it will be convenient to provide a proof of the
preparatory results of~\cite{X} in the language of~\cite{GGV1}.

We will first define approximate roots, final major roots and final minor roots using our language.

\subsection[Approximate $\pi$-roots]{Approximate $\boldsymbol{\pi}$-roots}
In this section we will consider a polynomial $P\in L$, which is monic in $y$.
For $l\in\mathds{N}$ we will consider the following algebras:
$$
L=K[x,y]\subsetneq K[x^{\pm \frac{1}{l}},y]\subsetneq K((x^{-1/l}))[y]\subsetneq K[\pi]((x^{-1/l}))[y],
$$
where $\pi$ is a variable (``symbol'' in~\cite{X}). We also will use the subring $L_\pi^{(l)}\coloneqq K[\pi][x^{\pm 1/l},y]$ of
$K[\pi]((x^{-1/l}))[y]$. Note that $\deg_x=v_{1,0}$ is well defined in $K[\pi]((x^{-1/l}))[y]$.

Unless otherwise indicated, we will consider the
elements $P$ of the above mentioned algebras as polynomials in $y$ with coefficients in one of the algebras $K[x]$, $K[x^{\pm \frac{1}{l}},y]$,
$K[\pi]((x^{-1/l}))$,\dots. Consequently expressions like $P(\tau)$, $P(\alpha)$,\dots, will denote $P$ with~$y$ replaced by $\tau$,
$P$ with~$y$ replaced by $\alpha$, etc.

By the Newton-Puiseux Theorem
(see~\cite{E}*{Corollary~13.15, page~295}) there exist $l\in\mathds{N}$ and $\alpha_i,\beta_i\in K((x^{-1/l}))$ such that
$$
P=\prod_{i=1}^{M}(y-\alpha_i).
$$
We set $\mathcal{R}(P)=\{\alpha_i: i=1,\dots,M\}$.
\begin{definition}\label{approximation up to}
Let $\alpha\in\mathcal{R}(P)$ and write $\alpha=\sum_j a_j x^j$ with $j\in\frac 1l\mathds{Z}$. The \emph{$\pi$-approximation of $\alpha$ up to $x^{j_0}$} is the element
$$
\tau\coloneqq\sum_{j>j_0}a_j x^j+\pi x^{j_0}\in K[\pi,x^{\pm\frac 1l}].
$$
Note that $\deg_x(\tau-\alpha)=j_0$.
\end{definition}

\begin{definition}\label{set of approximated roots}
Let $\tau\coloneqq\sum_{j>j_0}a_j x^j+\pi x^{j_0}\in K[\pi,x^{\pm\frac 1l}]$. We set
$$
D_{\tau}^{P}\coloneqq \{\alpha\in\mathcal{R}(P) : \tau\text{ is the $\pi$-approximation of $\alpha$ up to $x^{j_0}$}\}.
$$
If $\alpha\in D_\tau^P$ the we say that \emph{$\tau$ approximates $\alpha$ up to $x^{j_0}$}.
\end{definition}

\noindent Note that the element $\alpha_i=\sum_j b_jx^j\in \mathcal{R}(P)$ belongs to $D_{\tau}^{P}$ if and only if $\deg_x(\hat\alpha_i)\le j_0$, where
$\hat\alpha_i\coloneqq\alpha_i-\sum_{j>j_0}a_j x^j$, i.e. if and only if $a_j=b_j$ for all $j>j_0$.

\begin{definition}\label{pi-root}
We say that $\tau\coloneqq  \sum_{j>j_0}a_j x^j+\pi x^{j_0}\in K[\pi,x^{\pm\frac 1l}]$ is \emph{a $\pi$-root of $P$} if there exists $\alpha\in\mathcal{R}(P)$
such that $\pi$ approximates $\alpha$ up to $x^{j_0}$. We say that $j_0$ is \emph{the order of $\tau$}.
\end{definition}

\begin{notation}
Let $\tau\coloneqq \sum_{j>j_0}a_j x^j+\pi x^{j_0}$ be a $\pi$-root of $P$. We denote by $\varphi_\tau$ the automorphism
of $L^{(l)}$ given by $\varphi_\tau(x^{1/l})\coloneqq x^{1/l}$ and
$\varphi_\tau(y)\coloneqq y+\sum_{j>j_0}a_j x^j$.
\end{notation}

\begin{remark}\label{numero de raices aproximadas no crece}
Let $\alpha\in\mathcal{R}(P)$. Assume that $\tau$ approximates $\alpha$ up to $j_0$ and $\tau_1$ approximates $\alpha$ up to $j_1$. If $j_0>j_1$, then
$D_{\tau_1}^P\subseteq D_\tau^P$.
\end{remark}

In the sequel, for each $j\in\frac 1l\mathds{Z}$, we let $\dir(j)$ denote the unique direction $(\rho,\sigma)$ such that $\rho>0$ and
$j=\frac{\sigma}{\rho}$. Moreover, given a polynomial $\tau=\sum_{i>j_0} a_i x^i+\pi x^{j_0}$, we set $z\coloneqq x^{-\sigma/\rho}y$, where
$(\rho,\sigma)=\dir(j_0)$.

\smallskip

The following proposition shows that our definition of $\pi$-root coincides with the one given in~\cite{M}*{Definition~1.3}, with $x^{-1}$ replaced by $t$.

\begin{proposition}\label{proposition definition pi root}
Let $\tau=\sum_{j>j_0}a_j x^j+\pi x^{j_0}$ and let $f_{P,\tau}(\pi)\in K[\pi]$ be the polynomial determined by the equality
\begin{equation}\label{eq definition pi root}
P(\tau)=f_{P,\tau}(\pi)x^{\lambda_\tau}+\text{terms with lower order in $x$},
\end{equation}
where $\lambda_\tau\coloneqq \deg_x(P(\tau))\in\frac 1l \mathds{Z}$. Set $\varphi\coloneqq \varphi_\tau$ and $(\rho,\sigma)=\dir(j_0)$. We have
\begin{equation}\label{grado de f es numero de elementos}
|D_{\tau}^{P}|=\deg(f_{P,\tau})=v_{0,1}(\en_{\rho,\sigma}(\varphi(P))),
\end{equation}
and
\begin{equation}\label{igualdad para f de P y tau}
\ell_{\rho,\sigma}(\varphi(P))=x^{\lambda_\tau}f_{P,\tau}(z).
\end{equation}
Consequently $\tau$ is a $\pi$-root of $P$ if and only if $\deg(f_{P,\tau})>0$.
\end{proposition}

\begin{proof}
Let $\ev_{\pi x^{j_0}}\colon L_\pi^{(l)}\to L_\pi^{(l)}$ be the evaluation of $y$ in $\pi x^{j_0}$. So, $\ev_{\pi x^{j_0}}(y)\coloneqq \pi x^{j_0}=\pi
x^{\sigma/\rho}$ and $\ev_{\pi x^{j_0}}(x^{1/l})\coloneqq x^{1/l}$. Note that $P(\tau)=\ev_{\pi x^{j_0}}(\varphi(P))$. Since $\ev_{\pi x^{j_0}}$ is
$(\rho,\sigma)$-homogeneous,
$$
\ell_{\rho,\sigma}(\ev_{\pi x^{j_0}}(\varphi(P)))=\ev_{\pi x^{j_0}}(\ell_{\rho,\sigma}(\varphi(P))).
$$
On the other hand, since $\rho|l$,
\begin{equation}\label{ell de P}
\ell_{\rho,\sigma}(\varphi(P))=x^{r/l}g(z)\quad\text{for some $r\in\mathds{Z}$
 and $g(z)\in K[z]$.}
\end{equation}
Using that $\ev_{\pi x^{j_0}}(z)=\pi$, from this we obtain
$$
\ell_{\rho,\sigma}(\ev_{\pi x^{j_0}}(\varphi(P)))=x^{r/l}g(\pi).
$$
Note that
$$
P(\tau)=\ev_{\pi x^{j_0}}(\varphi(P))=x^{r/l}g(\pi)+\text{terms with lower order in x},
$$
because $v_{\rho,\sigma}(x^j)=j\rho< \rho r/l=v_{\rho,\sigma}(x^{r/l})$ if and only if $j<r/l$. So $f_{P,\tau}(\pi)= g(\pi)$,
$\lambda_\tau= r/l$, and equality~\eqref{ell de P} becomes equality~\eqref{igualdad para f de P y tau}. Since
$\deg_z(\ell_{\rho,\sigma}(\varphi(P)))=\deg_y(\ell_{\rho,\sigma}(\varphi(P)))$, we also have
$\deg(f_{P,\tau})=v_{0,1}(\en_{\rho,\sigma}(\varphi(P)))$. Consequently, in order to conclude the proof, it suffices to prove that
$|D_\tau^P|=v_{0,1}(\en_{\rho,\sigma}(\varphi(P)))$. Note that
$$
v_{0,1}(\en_{\rho,\sigma}(\varphi(P)))=\sum_{i=1}^M v_{0,1}(\en_{\rho,\sigma}(\varphi(y-\alpha_i)))=\sum_{i=1}^M v_{0,1}(\en_{\rho,\sigma}(y-\hat\alpha_i))),
$$
where $\hat\alpha_i=\alpha_i-\sum_{j>j_0}a_j x^j$. But
$$
\en_{\rho,\sigma}(y-\hat\alpha_i))=\begin{cases} (0,1)&\text{ if }\deg_x(\hat\alpha_i)\le \sigma/\rho=j_0,\\
(\deg_x(\hat\alpha_i),0)&\text{ if }\deg_x(\hat\alpha_i)> \sigma/\rho=j_0.
\end{cases}
$$
So
$$
\sum_{i=1}^M v_{0,1}(\en_{\rho,\sigma}(y-\hat\alpha_i)))=\#\{\alpha_i\in\mathcal{R}(P): \deg_x(\hat\alpha_i)\le j_0 \}=|D_\tau^P|,
$$
as desired.
\end{proof}

%From now on we write $P(\tau)$, $Q(\tau)$, $P(\alpha)$, \dots \ instead of $P(y=\tau)$, $Q(y=\tau)$, $P(\alpha=y)$, etc.
\begin{definition}\label{pi-root}
We say that a $\pi$-root $\tau$  of $P$ is a \emph{a final $\pi$-root of $P$} if $f_{P,\tau}(\pi)$ has no multiple roots and
$\deg_{\pi}(f_{P,\tau}(\pi)) > 1$, where $f_{P,\tau}(\pi)$ is defined by equality~\eqref{eq definition pi root}.
\end{definition}

\begin{remark} \label{rho sigma esta en dir de P}
Let $\tau$ be a final $\pi$-root of $P$. Since the support of $f_{P,\tau}$ has more than one point, from equality~\eqref{igualdad para f de P y tau}
it follows that $(\rho,\sigma)\in \Dir(\varphi_\tau(P))$.
\end{remark}

\begin{proposition}\label{factor genera pi raiz}
Let $\tau\coloneqq \sum_{j=1}^k a_jx^j+\pi x^{j_0}$ be a $\pi$-root of $P$ and let $\lambda\in K$. Consider the automorphism $\varphi_1\colon
L^{(l)}\to L^{(l)}$ given by $\varphi_1(x^{1/l})\coloneqq x^{1/l}$ and $\varphi_1(y)\coloneqq y+ \sum_{j=1}^k a_jx^j+\lambda x^{j_0}$. Assume
that $\varphi_1(P)$ is not a monomial and set $(\rho',\sigma')\coloneqq
\Pred_{\varphi_1(P)}(\rho,\sigma)$ (see~\cite{GGV1}*{Definition~3.4}), where $(\rho,\sigma)\coloneqq\dir(j_0)$. If $\rho'>0$, then set
$j_1=\frac{\sigma'}{\rho'}$, else take any $j_1\in\frac 1l \mathds{Z}$ with $j_1<j_0$. In both cases set $(\rho_1,\sigma_1)\coloneqq \dir(j_1)$.  If
$\pi-\lambda$ has multiplicity $r>0$ in $f_{P,\tau}(\pi)$, then
$$
\tau_1\coloneqq \sum_{j=1}^k a_jx^j+\lambda x^{j_0}+\pi x^{j_1}
$$
is a $\pi$-root of $P$ and $|D_{\tau_1}^P|=r$ (note that $j_1<j_0$). Moreover,
\begin{equation}
\label{direccion en el intervalo}
(\rho_1,\sigma_1)\in\! [\Pred_{\varphi_1(P)}(\rho,\sigma), (\rho,\sigma)[\hspace{0.5pt}.
\end{equation}
\end{proposition}

\begin{proof}
Write $\varphi_1=\widetilde \varphi\circ \varphi$, where $\varphi$ is as in Proposition~\ref{proposition definition pi root},
$\widetilde\varphi(y)=y+\lambda x^{j_0}$ and $\widetilde\varphi(x)=x$. By equality~\eqref{igualdad para f de P y tau} and the fact that
$\widetilde\varphi$ is $(\rho,\sigma)$-homogeneous and $\widetilde\varphi(z)=z+\lambda$, we have
$$
\ell_{\rho,\sigma}(\varphi_1(P))=\widetilde\varphi(\ell_{\rho,\sigma}(\varphi(P)))=\widetilde\varphi(x^{\lambda_\tau}f_{P,\tau}(z))
=x^{\lambda_\tau}\widetilde\varphi(f_{P,\tau}(z))=x^{\lambda_\tau}z^r g_1(z),
$$
for some $g_1(z)\in K[z]$ with $g_1(0)\ne 0$. By construction,  $(\rho_1,\sigma_1)\in\! [\Pred_{\varphi_1(P)}(\rho,\sigma), (\rho,\sigma)[$\hspace{0.5pt}, and
so, by Proposition~\ref{proposition definition pi root}, we have
$$
r=v_{0,1}(\st_{\rho,\sigma}(\varphi_1(P)))=v_{0,1}(\en_{\rho_1,\sigma_1}(\varphi_1(P)))=|D_{\tau_1}^P|,
$$
as desired.
\end{proof}

\begin{corollary}\label{lambda tiene que ser raiz}
Let $\tau\coloneqq \sum_{j=1}^k a_jx^j+\pi x^{j_0}$ be a $\pi$-root of $P$ and let $\lambda\in K$. If $\pi-\lambda$ does not divide
$f_{P,\tau}(\pi)$, then there exists no root $\alpha\in\mathcal{R}(P)$ such that $\deg_x(\alpha-(\lambda x^{j_0}+\sum_{j=1}^k a_jx^j))<j_0$.
\end{corollary}

\begin{proof}
Let $f_{P,\tau}(\pi)=\prod_{i=1}^k(\pi-\lambda_i)^{m_i}$. By Proposition~\ref{factor genera pi raiz} for each $i$ there exists $\tau_1(i)$ and $m_i$
roots in $D_{\tau_1(i)}^P\subset D_\tau^P$, for which $\Coeff_{x^{j_0}}=\lambda_i$. Since
$$
|D_\tau^P|=\deg(f_{P,\tau}(\pi))=\sum_{i=1}^{k}m_i=\sum_{i=1}^{k}|D_{\tau_1(i)}^P|,
$$
and the sets $D_{\tau_1(i)}^P$ are pairwise disjoint, we obtain $D_\tau^P=\bigcup_{i=1}^{k}D_{\tau_1(i)}^P$. Consequently, the coefficient of $x^{j_0}$ in each element of $D_\tau^P$ is a root of $f_{P,\tau}$. Since $\lambda$ is not a root of $f_{P,\tau}$, this finishes the proof.
\end{proof}

\begin{remark}\label{multiplicidad es cota para numero de raices mejor aproximadas}
The proof of the corollary shows that if the multiplicity of $\pi-\lambda$ in $f_{P,\tau}(\pi)$ is $r$, then any $\pi$-root $\tau_2$ of $P$ which
begins with $\lambda x^{j_0}+\sum_{j=1}^k a_jx^j$ satisfies $|D_{\tau_2}^P|\le r$.
\end{remark}

\begin{remark}\label{borde superior no nulo}
Let $\alpha\coloneqq\sum_j a_jx^j\in K((x^{-1/l}))$, $j_0\in\frac 1l \mathds{Z}$, $\tau\coloneqq \sum_{j>j_0}a_jx^j+\pi x^{j_0}$ and $(\rho,\sigma)\coloneqq \dir(j_0)$. Define
$T\coloneqq \sum_{j\le j_0}a_jx^j$.
Since
$$
P(\alpha)=\ev_{y=T}(\varphi_\tau(P)),
$$
we have
$\ell_{\rho,\sigma}(P(\alpha))=\ell_{\rho,\sigma}(\ev_{y=\lambda x^{j_0}}(\varphi_\tau(P)))$, whenever the right hand side of the equality is nonzero.
\end{remark}

\begin{proposition}
\label{P en tau coincide con P en alpha}
Let $\alpha=\sum_{j>j_0} a_j x^j+\lambda x^{j_0}+\sum_{j<j_0} a_j x^j$ and set
$\tau\coloneqq \sum_{j>j_0} a_j x^j+\pi x^{j_0}$. If $f_{P,\tau}(\lambda)\ne 0$, then
 $\lambda_\tau^P=\deg_x(P(\tau))=\deg_x(P(\alpha))$.
\end{proposition}

\begin{proof}
By Remark~\ref{borde superior no nulo}, equality~\eqref{igualdad para f de P y tau} and the fact that $\ev_{y=\lambda x^{j_0}}$ is $(\rho,\sigma)$-homogeneous, we have
$$
\ell_{\rho,\sigma}(P(\alpha))=\ell_{\rho,\sigma}(\ev_{y=\lambda x^{j_0}}(\varphi(P)))=\ev_{y=\lambda x^{j_0}}(\ell_{\rho,\sigma}(\varphi(P)))
=x^{\lambda_\tau^P}f_{P,\tau}(\lambda).
$$
Therefore
$$
\deg_x(P(\alpha)=\deg_x(\ell_{\rho,\sigma}(P(\alpha)))=\lambda_\tau^P=\deg_x(P(\tau)),
$$
as desired.
\end{proof}

\section{Approximate roots for Jacobian pairs}

\label{subseccion 21}
For the rest of the section we let $(P_0,Q_0)$ denote a Jacobian pair in $L$ satisfying the conditions required
in~\cite{GGV1}*{Corollary~5.21}, which in particular means that $(P_0,Q_0)$ is a minimal pair and a standard $(m,n)$-pair for
some coprime integers $m, n >1$. By~\cite{GGV1}*{Proposition~4.6(3)},
there exist $a<b$ in $\mathds{N}$ such that $\en_{1,0}(P_0)=m(a,b)$  and $\en_{1,0}(Q_0)=n(a,b)$.
So, by~\cite{GGV1}*{Corollary~5.21(4)}, we know that $\ell_{1,1}(P_0) = \lambda
x^{am}y^{bm}$ and $\ell_{1,1}(Q_0) = \lambda' x^{an}y^{bn}$ for some
$\lambda,\lambda'\in K^{\times}$. Replacing $P_0$ by $\frac{1}{\lambda}P_0$ and $Q_0$ by $\frac{1}{\lambda'}Q_0$,
we can assume that $\lambda=\lambda'=1$.
Let {$\psi$ be the automorphism of~$L$ defined by}
$\psi(y)\coloneqq y$ and $\psi(x)\coloneqq x+y$, and set $P\coloneqq \psi(P_0)$ and $Q\coloneqq \psi(Q_0)$ (see Figure~\ref{figura P cero}).
Since $\psi$ is $(1,1)$-homogeneous,
\begin{equation}\label{esquina superior}
\ell_{1,1}(P)=\psi(\ell_{1,1}(P_0))=(x+y)^{ma}y^{mb}\quad\text{and}\quad \ell_{1,1}(Q)=\psi(\ell_{1,1}(Q_0))=(x+y)^{na}y^{nb}.
\end{equation}
Hence, $P$ and $Q$ are monic polynomials in $y$ and moreover, a straightforward computation shows that
\begin{equation}\label{esquina superior uno cero}
\en_{1,0}(P)=m(a,b)\quad\text{and}\quad \en_{1,0}(Q)=m(a,b).
\end{equation}

\begin{figure}[htb]
\centering
\begin{tikzpicture}
\fill[gray!20] (0,0) -- (0.25,0.75) -- (1,2.5) -- (1,1.25) -- (0.5,0) -- (0,0);
\draw  [thick] (0.25,0.75) -- (1,2.5) -- (1,1.25);
\draw[step=.5cm,gray,very thin] (0,0) grid (2,3.6);
\draw [->] (0,0) -- (2.3,0) node[anchor=north]{$x$};
\draw [->] (0,0) --  (0,4) node[anchor=east]{$y$};
\draw[dotted] (0,0) -- (2.3,2.3);
\draw[dotted] (0,1.5) -- (1,2.5);
\draw [->,thick] (1.5,2.5) -- (2,2.5) node[fill=white,anchor=west]{\tiny{(1,0)}};
\draw [->,thick] (0.5,2.5) -- (-0.5,3) node[fill=white,anchor=east]{\tiny{$\Succ_{P_0}(1,0)$}};
\fill[gray!20] (6,0) -- (6,3.5) -- (7,2.5) -- (7,1.25) -- (6.5,0) -- (6,0);
\draw  [thick] (6,3.5) -- (7,2.5) -- (7,1.25);
\draw[step=.5cm,gray,very thin] (6,0) grid (8,3.6);
\draw [->] (6,0) -- (8.3,0) node[anchor=north]{$x$};
\draw [->] (6,0) --  (6,4) node[anchor=east]{$y$};
\draw[dotted] (6,0) -- (8.3,2.3);
%\draw[dotted] (6,1.5) -- (7,2.5);
%
\draw [->,thick] (7.5,2.5) -- (8,2.5) node[fill=white,anchor=west]{\tiny{(1,0)}};
\draw[->] (3,2) .. controls (4,2.25) .. (5,2);
\draw (3.8,2.7) node[fill=white,anchor=west]{\Large$\psi$};
\draw (6.2,2) node[fill=white,anchor=west]{\Large$P$};
\draw [->,thick] (7,3) -- (7.5,3.5) node[fill=white,anchor=west]{\tiny{$(1,1)$}};
\end{tikzpicture}
\caption{The shapes of $P_0$ according to~\cite{GGV1}*{Corollary~5.21(4)} and of $P$ according to~\eqref{esquina superior} and~\eqref{esquina superior uno cero}.}
\label{figura P cero}
\end{figure}
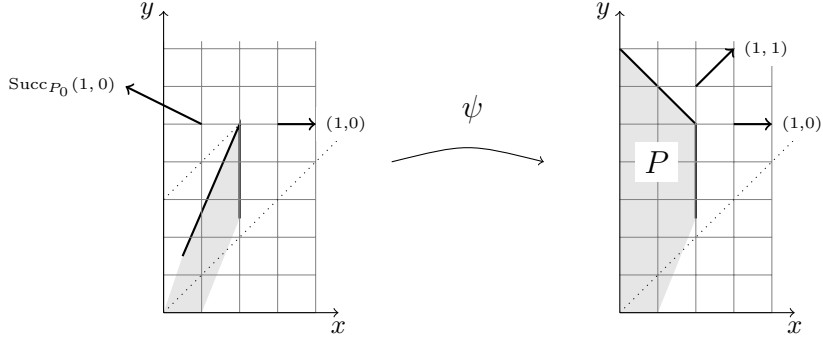

\begin{remark}\label{nm par}
In the sequel we will establish several results about $P$, but, since by~\cite{GGV1}*{Proposition~4.6} we know that $(Q,P)$ is an $(n,m)$-pair,
the same results are valid mutatis mutandis for~$Q$.
\end{remark}

\begin{proposition}\label{multiple roots}
Let $\alpha\in\mathcal{R}(P)$ and let $\tau$ be the $\pi$-approximation of $\alpha$ up to $x^{j_0}$. Assume that $\lambda_\tau\coloneqq
\deg_x(P(\tau))>0$, and take $\varphi$ and $(\rho,\sigma)$ as in Proposition~\ref{proposition definition pi root}. The following facts hold:
\begin{enumerate}

\item If $f_{P,\tau}$ has multiple roots, then $[\ell_{\rho,\sigma}(\varphi(P)),\ell_{\rho,\sigma}(\varphi(Q))]=0$.

\smallskip

\item If $[\ell_{\rho,\sigma}(\varphi(P)),\ell_{\rho,\sigma}(\varphi(Q))]=0$, then there exists $\beta\in \mathcal{R}(Q)$ such that $\deg_x(\alpha-\beta)<j_0$.
\end{enumerate}
\end{proposition}

\begin{proof}
Write $\tau\coloneqq \sum_{j=1}^k a_jx^j+\pi x^{j_0}$. By~Proposition~\ref{factor genera pi raiz} there exists $j_1<j_0$ such that
$$
\tau_1\coloneqq \sum_{j=1}^k a_jx^j+\lambda x^{j_0}+\pi x^{j_1}
$$
is a $\pi$-root of $P$. Now we prove the statements~(1) and (2).

\smallskip

\noindent(1)\enspace Since $\ell_{\rho,\sigma}(\varphi(P))=x^{\lambda_\tau}f_{P,\tau}(z)$ (see equality~\eqref{igualdad para f de P y tau}), by
hypothesis there exist $k>1$ and $\lambda\in K$ such that $(z-\lambda)^k$ divides $\ell_{\rho,\sigma}(\varphi(P))$. Consequently $(z-\lambda)^{k-1}$
divides $[\ell_{\rho,\sigma}(\varphi(P)),\ell_{\rho,\sigma}(\varphi(Q))]$. Since $[\ell_{\rho,\sigma}(\varphi(P)),\ell_{\rho,\sigma}(\varphi(Q))]\in
K$, this implies that $[\ell_{\rho,\sigma}(\varphi(P)),\ell_{\rho,\sigma}(\varphi(Q))]=0$.

\smallskip

\noindent(2)\enspace Let $(z-\lambda)$ be a linear factor of $\ell_{\rho,\sigma}(\varphi(P))$. Since $v_{\rho,\sigma}(P)=\rho \lambda_\tau>0$,
from~\cite{GGV1}*{Proposition~2.1(2)b)} it follows that $(z-\lambda)$ divides $\ell_{\rho,\sigma}(\varphi(Q))$. Hence, by Proposition~\ref{proposition
definition pi root} we know that $\tau$ is a $\pi$-root of $Q$ and so, by~Proposition~\ref{factor genera pi raiz}, there exists $j_2<j_0$ such that
$$
\tau_2\coloneqq \sum_{j=1}^k a_jx^j+\lambda x^{j_0}+\pi x^{j_2}
$$
is a $\pi$-root of $Q$. From this it follows immediately that for any $\alpha\in D_{\tau_1}^P$ and $\beta\in D_{\tau_2}^Q$ the inequality
$\deg_x(\alpha-\beta)<j_0$ holds, as desired.
\end{proof}

\begin{remark}\label{orden de los lambda}
Let $\alpha=\sum a_j x^j \in \mathcal{R}(P)$. Assume that $j_0>j_1$, $\tau$ approximates $\alpha$ up to $x^{j_0}$ and  $\tau_1$ approximates $\alpha$
up to $x^{j_1}$. Then $\lambda_\tau> \lambda_{\tau_1}$. In fact, setting $(\rho,\sigma)\coloneqq\dir(j_0)$ and $(\rho_1,\sigma_1)\coloneqq\dir(j_1)$,
equality~\eqref{igualdad para f de P y tau} and~\cite{GGV1}*{Proposition~3.9} show that
$$
v_{\rho,\sigma}(x^{\lambda_\tau})=v_{\rho,\sigma}(\varphi(P))=v_{\rho,\sigma}(\varphi_1(P))\ge v_{\rho,\sigma}(\en_{\rho_1,\sigma_1}\varphi_1(P)),
$$
where $\varphi\coloneqq \varphi_{\tau}$ and $\varphi_1\coloneqq \varphi_{\tau_1}$.
Moreover, a direct computation using that $(\rho,\sigma)>(\rho_1,\sigma_1)$, $v_{\rho_1,\sigma_1}(\varphi_1(P))=
v_{\rho_1,\sigma_1}(x^{\lambda_{\tau_1}})$ and $v_{0,1}(\en_{\rho_1,\sigma_1}(\varphi_1(P)))>v_{0,1}(x^{\lambda_{\tau_1}})$, shows that
$$
v_{\rho,\sigma}(\en_{\rho_1,\sigma_1}(\varphi_1(P)))>v_{\rho,\sigma}(x^{\lambda_{\tau_1}}).
$$
Since $\rho>0$, this proves that $\lambda_\tau>
\lambda_{\tau_1}$.
\end{remark}

\begin{proposition}\label{approximation with lambda equal to 0}
Let $\alpha\in\mathcal{R}(P)$. There exists $j_0$ such that $\lambda_\tau=0$ for the $\pi$-approximation $\tau$ of $\alpha$ up to $x^{j_0}$.
\end{proposition}

\begin{proof}
Let $\varphi_0\in\Aut(K((x^{-1/l}))[y])$ be given by $\varphi_0(x^{1/l})=x^{1/l}$ and $\varphi_0(y)=y+\alpha$. We will construct a direction
$(\rho_0,\sigma_0)\in ](0,-1),(0,1)[$ such that $v_{\rho_0,\sigma_0}(\varphi_0(P))=0$. In order to do this, for each point of $\Supp(\varphi_0(P))$,
we consider the direction $(\rho,\sigma)\in ](0,-1),(0,1)[$ orthogonal to the line that passes through that point and through the origin. The minimum
$(\rho_0,\sigma_0)$ of these directions satisfies $v_{\rho_0,\sigma_0}(\varphi_0(P))=0$. Set $j_0\coloneqq \frac{\sigma_0}{\rho_0}$. We assert that
the $\pi$-approximation
$$
\tau=\sum_{j>j_0}a_j x^j+\pi t^{j_0}
$$
of $\alpha$ up to $x^{j_0}$, satisfies $\lambda_\tau=0$. In fact, we have
$$
0=v_{\rho_0,\sigma_0}(\varphi_0(P))=v_{\rho_0,\sigma_0}(\varphi_\tau(P))=v_{\rho_0,\sigma_0}(x^{\lambda_\tau})=\rho_0 \lambda_\tau,
$$
where  the second equality follows using~\cite{GGV1}*{Proposition~3.9} and the third equality, from~\eqref{igualdad para f de P y tau}.
\end{proof}

\begin{proposition}\label{common pi roots}
Let $\tau\coloneqq \sum_{j=1}^k a_jx^j+\pi x^{j_0}$ be a $\pi$-root of $P$, and let $(\rho,\sigma)$, $\lambda_\tau$ and $\varphi$ be as in
Proposition~\ref{proposition definition pi root}. If $\tau$ is also a $\pi$-root of $Q$ and $\lambda_\tau\ge 0$, then
$$
\en_{\rho,\sigma}(\varphi(Q))=\frac nm \en_{\rho,\sigma}(\varphi(P))\quad\text{and}\quad \frac{|D_{\tau}^Q|}{|D_{\tau}^P|}=\frac nm.
$$
\end{proposition}

\begin{proof}
Write $\Dir(\varphi(P))\cap[(\rho,\sigma),(1,1)]=\{(\rho,\sigma)=(\rho_0,\sigma_0)<(\rho_1,\sigma_1)<\dots <(\rho_k,\sigma_k)=(1,1)\}$. Take
$\alpha\in D_\tau^P$ and $0\le i\le k$. Let $j_i\coloneqq \frac{\sigma_i}{\rho_i}$ and let $\tau_i$ be the $\pi$-approximation of $\alpha$ up to
$x^{j_i}$. Set $\lambda_{\tau_i}\coloneqq \deg_x(P(\tau_i))$ and $\varphi_i\coloneqq \varphi_{\tau_i}$.
Since
$$
[\ell_{\rho_i,\sigma_i}(\varphi(P)),\ell_{\rho_i,\sigma_i}(\varphi(Q))]\in K,
$$
if $[\ell_{\rho_i,\sigma_i}(\varphi(P)),\ell_{\rho_i,\sigma_i}(\varphi(Q))]\ne 0$, then $v_{0,-1}([\ell_{\rho_i,\sigma_i}(\varphi(P)),
\ell_{\rho_i,\sigma_i}(\varphi(Q))])=0$ and then, by~\cite{GGV1}*{Proposition~1.13},
$$
0=v_{0,-1}([\ell_{\rho_i,\sigma_i}(\varphi(P)),\ell_{\rho_i,\sigma_i}(\varphi(Q))])\le v_{0,-1}(\ell_{\rho_i,\sigma_i}(\varphi(P)))+
v_{0,-1}(\ell_{\rho_i,\sigma_i}(\varphi(Q)))-(-1+0),
$$
which implies that
$$
v_{0,-1}(\st_{\rho_i,\sigma_i}(\varphi (P))+v_{0,-1}(\st_{\rho_i,\sigma_i}(\varphi (Q))\ge -1,
$$
or, equivalently,
$$
v_{0,1}(\st_{\rho_i,\sigma_i}(\varphi (P))+v_{0,1}(\st_{\rho_i,\sigma_i}(\varphi (Q))\le 1.
$$
This implies that $i=0$. Hence, if $i>0$, then $[\ell_{\rho_i,\sigma_i}(\varphi(P)),\ell_{\rho_i,\sigma_i}(\varphi(Q))]=0$, and since by
Remark~\ref{orden de los lambda} we know that $\lambda_{\tau_i}>0$, we have
$$
v_{\rho_i,\sigma_i}(\varphi(P))=v_{\rho_i,\sigma_i}(\varphi_i(P))=\rho_i \lambda_{\tau_i}>0,
$$
where the first equality follows from~\cite{GGV1}*{Proposition~3.9} and the second one from~\eqref{igualdad para f de P y tau}. Now, an inductive
argument using~\eqref{esquina superior}, \cite{GGV1}*{Remark~3.1} and that $\en_{\rho_i,\sigma_i}(\varphi(P)))=
\st_{\rho_{i+1},\sigma_{i+1}}(\varphi(P)))$ for $i=k,\dots,1$, proves that
$$
v_{\rho_i,\sigma_i}(\varphi(Q))>0\quad\text{and}\quad\st_{\rho_i,\sigma_i}(\varphi(Q))=\frac nm \st_{\rho_i,\sigma_i}(\varphi(P)),
\quad\text{for $i=k,\dots,1$.}
$$
for $i=k,\dots,1$. Hence
$$
\en_{\rho_0,\sigma_0}(\varphi(Q))=\frac nm \en_{\rho_0,\sigma_0}(\varphi(P))\quad\text{and}\quad
\frac{v_{0,1}(\en_{\rho_0,\sigma_0}(\varphi(Q)))}{v_{0,1}(\en_{\rho_0,\sigma_0}(\varphi(P)))}=\frac nm.
$$
This finishes the proof, since $\frac{|D_{\tau}^Q|}{|D_{\tau}^P|}= \frac{v_{0,1}(\en_{\rho,\sigma}(\varphi(Q)))}{v_{0,1}
(\en_{\rho,\sigma}(\varphi(P)))}$ by Proposition~\ref{proposition definition pi root} and $(\rho_0,\sigma_0)=(\rho,\sigma)$.
\end{proof}

In~\cite{X} the author chooses a generic element $\xi\in K$ and analyses the roots of $P_\xi=P+\xi$. Instead of speaking of a generic element $\xi$,
we will assume (adding eventually to $P$ an element $\xi\in K$) that for any $\pi$-root $\tau$ of $P$ with $\lambda_\tau=0$ we have
\begin{itemize}

\item[(1)] $f_{P,\tau}$ has no multiple roots.

\smallskip

\item[(2)] $f_{P,\tau}$ and $f_{Q,\tau}$ have no common roots (are coprime).
\end{itemize}
This is possible, since, by~\eqref{igualdad para f de P y tau}, in the case $\lambda_\tau=0$ adding $\xi$ to $P$ is the same as adding $\xi$ to the
univariate polynomial $f_{P,\tau}(z)$. We also can and will assume that $(0,0)\in\Supp(P)\cap \Supp(Q)$.

\begin{remark}\label{raiz con lambda negativo}
Assume that $\tau$ is a $\pi$-root of $P$ with $\lambda_\tau<0$. Then, by Proposition~\ref{factor genera pi raiz}, Remark~\ref{numero de raices
aproximadas no crece} and item~(1), we have $|D_{\tau}^P|=1$. Moreover, we also have $|D_{\tau}^Q|=0$. In fact, take $\alpha\in D_\tau^P$.
By Proposition~\ref{approximation with lambda equal to 0} there exists $j_1$ and a $\pi$-approximation $\tau_1$ of $\alpha$ up to $x^{j_1}$, such that
$\lambda_{\tau_1}=0$. By Remark~\ref{orden de los lambda} necessarily $j_1>j_0$, where $j_0$ is the order of $\tau$. Let $\lambda$ be the coefficient
of $\alpha$ at $x^{j_1}$. Then $\pi-\lambda|f_{P,\tau_1}$ and so, by item~(2), $\pi-\lambda\nmid f_{Q,\tau_1}$. If $\tau_1$ is not a $\pi$-root of
$Q$, then clearly $|D_{\tau}^Q|=0$. Otherwise, by Corollary~\ref{lambda tiene que ser raiz} applied to $\tau_1$ and $Q$, we also have
$|D_{\tau}^Q|=0$.
\end{remark}

\begin{remark}\label{raiz final tiene lambda positivo}
From the first assertion in the previous remark it follows that for any final $\pi$-root $\tau$ of $P$ we have $\lambda_\tau\ge 0$.
\end{remark}

\begin{notation}\label{delta de alpha}
Let $\alpha=\sum_j a_j x^j\in \mathcal{R}(P)$  and set
$\delta_\alpha\coloneqq \min\{\deg_x(\alpha - \beta) | \beta\in\mathcal{R}(Q)\}$.
\end{notation}

\begin{remark}\label{pi raiz final es pi raiz de Q}
The $\pi$-ap\-proxi\-mation of $\alpha$ up to $x^{\delta_\alpha}$ is also a $\pi$-root of $Q$.
\end{remark}

\begin{proposition}(\cite{X}*{Lemma~4.2})\label{raiz final de una raiz}
Set $\tau\coloneqq \sum_{j>\delta}a_j x^j+\pi x^{\delta_\alpha}$. Then $\tau$ is a final $\pi$-root of~$P$.
\end{proposition}

\begin{proof}
Since clearly $\tau$ is a $\pi$-root of $P$, we only must prove that $\tau$ is
a final $\pi$-root of $P$, i.e, that $\deg(f_{P,\tau})>1$ and that $f_{P,\tau}$ has no multiple roots. By Remark~\ref{raiz con lambda negativo} we
know that $\lambda_\tau\ge 0$. By item~(1) above the same remark we also know that when $\lambda_\tau=0$, the polynomial $f_{P,\tau}$ has no multiple
roots. If $\lambda_\tau>0$, then $f_{P,\tau}$ also does not have no multiple roots. In fact, otherwise by Proposition~\ref{multiple roots} there
exists $\beta\in\mathcal{R}(Q)$ such that $\deg_x(\alpha-\beta)<\delta_\alpha$, contradicting the definition of $\delta_\alpha$. Finally, by Proposition~\ref{common
pi roots} we know that $m$ divides $|D_\tau^P|=\deg(f_{P,\tau})$ and so $\deg(f_{P,\tau})>1$, which concludes the proof.
\end{proof}

\subsection[Major and minor final $\pi$-roots]{Major and minor final $\boldsymbol{\pi}$-roots}

\begin{definition}\label{major and minor roots}
A final $\pi$-root $\tau$ of $P$ is called a \emph{minor final $\pi$-root of $P$} if $\lambda_\tau=0$, and it is called a \emph{major final $\pi$-root of $P$} if $\lambda_\tau>0$. The set of minor final $\pi$-roots of $P$ is denoted by $P_m$ and the set of final major $\pi$-roots of $P$ is denoted by $P_M$.
\end{definition}
Note that
$$
\mathcal{R}(P)=\bigcup_{\tau\in P_m\cup P_M} D_\tau^P,
$$
since, by Proposition~\ref{raiz final de una raiz} every root $\alpha\in\mathcal{R}(P)$ is associated with a final $\pi$-root of $P$
(that we will call \emph{the final $\pi$-root of $P$ associated with $\alpha$}) and by
Remark~\ref{raiz final tiene lambda positivo} we know that $\lambda_\tau\ge 0$. Note also that
if $\tau\ne \tau_1$ are final $\pi$-roots, then $D^P_\tau\cap D^P_{\tau_1}=\emptyset$. In fact, assume by contradiction that $\alpha\in D^P_\tau\cap D^P_{\tau_1}$,
 and assume for example that $\delta_\tau<\delta_{\tau_1}$, which means that $\tau$ is a better approximation of $\alpha$. Then,
  since the multiplicity of
any factor of $f_{P,\tau_1}$ is one, by Remark~\ref{multiplicidad es cota para numero de raices mejor aproximadas} we have $|D_\tau^P|\le 1$,  which contradicts the fact that $\tau$ is a final $\pi$-root of $P$.

\begin{remark}\label{deltas son iguales}
Given a final $\pi$-root $\tau$ of $P$ take $\alpha\in D_\tau^P$. Then, by Proposition~\ref{raiz final de una raiz}, the $\pi$-ap\-proxi\-ma\-tion of $\alpha$ up to $x^{\delta_\alpha}$ is a final $\pi$-root, and, since $D^P_\tau\cap D^P_{\tau_1}=\emptyset$ for any other final $\pi$-root $\tau_1$ of $P$, necessarily
$\tau$ is the $\pi$-ap\-proxi\-ma\-tion of $\alpha$ up to $x^{\delta_\alpha}$, and so $\delta_\tau=\delta_\alpha$.
\end{remark}

\begin{proposition}
\label{raices menores y raices mayores}
Let $\tau$ be a final $\pi$-root of $P$, let $\varphi\coloneqq \varphi_\tau$ and set $\lambda_\tau^Q\coloneqq \deg_x(Q(\tau))$.
The following facts hold:
\begin{enumerate}
  \item If $\tau$ is a minor final $\pi$-root of $P$, then
  \begin{enumerate}
    \item[a)]$\lambda_\tau^Q=0$,
    \item[b)] $[\ell_{\rho,\sigma}(\varphi(Q)),\ell_{\rho,\sigma}(\varphi(P))]= 0$,
    \item[c)] $\delta_\tau<-1$.
  \end{enumerate}
  \item If $\tau$ is a major final $\pi$-root of $P$, then
\begin{enumerate}
\item[a)] $[\ell_{\rho,\sigma}(\varphi(Q)),\ell_{\rho,\sigma}(\varphi(P))]\ne 0$,
\item[b)] $\tau$ is a major final $\pi$ root of $Q$,
\item[c)] $\lambda_\tau^Q=\frac nm deg_x(P(\tau))$,
\item[d)] $\delta_\tau>-1$.
\end{enumerate}
\end{enumerate}
\end{proposition}
\begin{proof}
By Remarks~\ref{pi raiz final es pi raiz de Q} and~\ref{deltas son iguales},
any final $\pi$-root $\tau$ of $P$ is also a $\pi$-root of $Q$. We will use this fact in the proofs of~(1)a) and~(2)b).

\noindent(1)\enspace  By
Proposition~\ref{common pi roots}, since $\lambda_\tau\ge 0$, we have
$m\en_{\rho,\sigma}(\varphi(Q))= n\en_{\rho,\sigma}(\varphi(P))$, and so
$$
\rho\lambda_\tau^Q=v_{\rho,\sigma}(\varphi(Q))=\frac nm v_{\rho,\sigma}(\varphi(P))=\frac nm \rho \lambda_\tau^P=0,
$$
where the first and third equality follow from~\eqref{igualdad para f de P y tau}. This implies that $\lambda_\tau^Q=\deg_x(Q(\tau))=0$, proving~a).
Moreover, by~\cite{GGV1}*{Proposition~2.1(1)} the vanishing of $v_{\rho,\sigma}(\varphi(Q))$ and $v_{\rho,\sigma}(\varphi(P))$ implies that
$[\ell_{\rho,\sigma}(\varphi(Q)),\ell_{\rho,\sigma}(\varphi(P))]= 0$, proving item~b). Now assume by contradiction that
$\frac{\sigma}{\rho}=\delta_\tau\ge -1$, which implies that $\rho+\sigma\ge 0$. Then, by~\cite{GGV1}*{Proposition~1.13}, we have
$$
0=v_{\rho,\sigma}([\varphi(P),\varphi(Q)])\le v_{\rho,\sigma}(\varphi(Q))+v_{\rho,\sigma}(\varphi(P))-(\rho,\sigma)=-(\rho+\sigma)\le 0,
$$
so we have equality and, again by~\cite{GGV1}*{Proposition~1.13}, we have $[\ell_{\rho,\sigma}(\varphi(Q)),\ell_{\rho,\sigma}(\varphi(P))]\ne 0$.
But this contradicts item~b) and thus proves $\delta_\tau<-1$, which is~c).

\noindent (2)\enspace By Remarks~\ref{pi raiz final es pi raiz de Q} and~\ref{deltas son iguales}, we know that $\tau$ is a $\pi$-root of $Q$
and, that for any $\alpha\in D_{\tau}^P$,
$$
\delta_\tau= \min\{\deg_x(\alpha - \beta) | \beta\in \mathcal{R}(Q)\}.
$$
Hence, by Proposition~\ref{multiple roots}(2), we have $[\ell_{\rho,\sigma}(\varphi(Q)),\ell_{\rho,\sigma}(\varphi(P))]\ne 0$, which proves~a). Moreover, by
Proposition~\ref{multiple roots}(1) with $Q$ and $P$ interchanged, $f_{Q,\tau}$ has no multiple roots. On the other hand, by
Proposition~\ref{common pi roots}, we have
$$
|D_{\tau}^Q|=\frac nm |D_\tau^P|>1,
$$
and so $\tau$ is a final $\pi$-root of $Q$. Again by Proposition~\ref{common pi roots} and equality~\eqref{igualdad para f de P y tau}, we have
$$
\rho \deg_x Q(\tau)=\rho \lambda_\tau^Q=v_{\rho,\sigma}(\varphi(Q))=\frac nm v_{\rho,\sigma}(\varphi(P))=\frac nm \rho \lambda_\tau^P
=\rho \frac nm \deg_x(P(\tau)),
$$
and so $\deg_x Q(\tau)=\frac nm \deg_x P(\tau))>0$, which finishes the proof of~b) and~c).
It remains to check that $\delta_\tau>-1$. Assume by contradiction that $\frac{\sigma}{\rho}=\delta_\tau\le -1$. Then $\rho+\sigma\le 0$, and so
$$
 v_{\rho,\sigma}(\varphi(Q))+v_{\rho,\sigma}(\varphi(P))-(\rho+\sigma)\ge \rho\lambda_\tau^P\left(1+\frac nm\right)>0=v_{\rho,\sigma}[\varphi(P),\varphi(Q)],
$$
which, by~\cite{GGV1}*{Proposition~1.13}, implies that $[\ell_{\rho,\sigma}(\varphi(Q)),\ell_{\rho,\sigma}(\varphi(P))]=0$. This contradicts
item~a) finishing the proof of item~d).
\end{proof}
\subsection{Intersection number and major roots}

\begin{lemma}\label{Q en alpha es Q en tau}
Let $\tau$ be a final $\pi$-root of $P$. Then $\lambda_\tau^Q\coloneqq\deg_x(Q(\tau))=\deg_x(Q(\alpha))$ for $\alpha\in D_\tau^P$.
\end{lemma}

\begin{proof}
We assert that $f_{P,\tau}(z)$ and $f_{Q,\tau}(z)$ have no common roots. In fact, assume on the contrary that $z-\lambda$ is a common factor. If $\tau$ is a major final root, then
$$
z-\lambda\mid [\lambda_\tau f_{P,\tau}(z),\lambda_\tau^Q f_{Q,\tau}(z)]=[\ell_{\rho,\sigma}(\varphi(Q)),\ell_{\rho,\sigma}(\varphi(P))]\in K^{\times},
$$
a contradiction; whereas, if $\tau$ is a minor root, then the choice of $\xi$ guarantees that
$f_{P,\tau}$ and $f_{Q,\tau}$ have no common roots.

Note that if the coefficient of $x^{j_0}$ in $\alpha$ is $\lambda$, then $f_{P,\tau}(\lambda)=0$, since otherwise $\pi-\lambda$ does not divide $f_{P,\tau}(\pi)$ and Corollary~\ref{lambda tiene que ser raiz}
leads to a contradiction. Hence, by the assertion $f_{Q,\tau}(\lambda)\ne 0$, and so, by Proposition~\ref{P en tau coincide con P en alpha},
we have $\deg_x(Q(\tau))=\deg_x(Q(\alpha))$
as desired.
\end{proof}

\begin{theorem}\label{Interseccion con raices mayores}
Set $I_M=\sum_{\tau\in P_M}|D_\tau^{P_\xi}|\lambda_\tau^Q$. Then $I_M=I(P,Q)$.
\end{theorem}

\begin{proof}
It is well known that $\Res_y(P,Q)=\prod_{\alpha\in \mathcal{R}(P)}Q(\alpha)$. Hence,
$$
I(P,Q)=\deg_x \prod_{\alpha\in \mathcal{R}(P)}Q(\alpha)=\sum_{\alpha\in \mathcal{R}(P)}\deg_x(Q(\alpha))=
\sum_{\tau\in P_m\cup P_M}\sum_{\alpha\in D_\tau^{P}}\deg_x(Q(\alpha)).
$$
By Lemma~\ref{Q en alpha es Q en tau},
$$
I(P,Q)=\sum_{\tau\in P_m\cup P_M}\sum_{\alpha\in D_\tau^{P}}\deg_x(Q(\alpha))=\sum_{\tau\in P_M}|D_\tau^{P_\xi}|\lambda_\tau^Q +\sum_{\tau\in
P_m}|D_\tau^{P_\xi}|\lambda_\tau^Q=\sum_{\tau\in P_M}|D_\tau^{P_\xi}|\lambda_\tau^Q,
$$
since $\lambda_\tau^Q=0$ if $\tau\in P_m$.
\end{proof}

\begin{definition}
A root $\alpha\in \mathcal{R}(P)$ is called a \emph{minor root}, if the associated final $\pi$-root $\tau$ is a minor final $\pi$-root, and
it is called a \emph{major root}, if $\tau$ is a major final $\pi$-root.
\end{definition}

\begin{proposition}
Let $\tau$ be an approximate $\pi$-root of $P$ of order $j_0\le 0$ with $\lambda_\tau\ge 0$ and let $(\rho,\sigma)\coloneqq\dir(j_0)$. If
$v_{1,-1}(\en_{\rho,\sigma}(\varphi_\tau(P)))>0$,
then any root $\alpha\in D_\tau^P$ is a minor root.
\end{proposition}

\begin{proof}
The hypotheses guarantee  that $(\varphi_\tau(P),\varphi_\tau(Q))$ and $(\rho,\sigma)$
satisfy the hypotheses of Proposition~\ref{jacobiano se anula} (for instance $(\rho,\sigma)\in\ ](0,-1),(1,0)]$, because $j_0\le 0$).
If $v_{\rho,\sigma}(\varphi_\tau(P))=\rho \lambda_\tau=0$, then $\tau$ is a minor final $\pi$-root and the result is true.
Else $v_{\rho,\sigma}(\varphi_\tau(P))=\rho \lambda_\tau>0$, since $\lambda_\tau\ge 0$. Take $\alpha\in D_\tau^P$.  By
Proposition~\ref{raices menores y raices mayores} it suffices to prove that $\delta_\alpha<-1$.
By Propositions~\ref{jacobiano se anula} and~\ref{multiple roots} we have $\delta_\alpha<\delta_\tau=j_0$, so the result is clear when $\delta_\tau\le -1$.
Assume that
$\delta_\tau > -1$. In this case $\rho+\sigma>0$, and using Proposition~\ref{direcciones positivas} and equality~\eqref{igualdad para f de P y tau} we
conclude that
$f_{P,\tau}(z)=\varsigma(z-\mu)^{mb}$ for some $\varsigma,\mu\in K^{\times}$, where
$b\coloneqq\frac 1m v_{0,1}(\en_{\rho,\sigma}(\varphi_\tau(P)))=\frac{|D_\tau^P|}{m}$ (see Proposition~\ref{proposition definition pi root}). Hence, by
Proposition~\ref{factor genera pi raiz} there exists $j_1<j_0$ such that
for the $\pi$-root
$$
\tau_1\coloneqq \sum_{j=1}^k a_jx^j+\mu x^{j_0}+\pi x^{j_1},
$$
we have $D_{\tau_1}^P=D_\tau^P$. If $j_1\le -1$, then we finish the proof immediately applying the above argument with $\tau$ replaced by $\tau_1$,
since $\lambda_{\tau_1}\ge 0$
(in fact, if $\lambda_{\tau_1}<0$, then by Remark~\ref{raiz con lambda negativo}, we have $|D_{\tau_1}^P|=1$, which is impossible because $bm=|D_\tau^P|$).
Assume now that $j_1 > -1$ and set $(\rho_1,\sigma_1)\coloneqq \dir(j_1)$.
By Proposition~\ref{direcciones positivas} we know that $\rho_1|l$, and so $j_1\in\frac 1l\mathds{Z}$. Hence,
if $j_0=-\frac kl$ for some $0\le k\le l$,  then $-j_1\in\{\frac{k+1}{l},\frac{k+2}{l},\dots,\frac{l-1}l,\frac{l}{l}\}$, so after repeating the same procedure a finite number $t$ of times, we arrive at $\delta_\alpha<j_t\le -1$, as desired.
\end{proof}

\begin{proposition}\label{top roots}
Let $a,b$ satisfying equalities~\eqref{esquina superior}. There exist $ma$ minor roots $\alpha$ of $P$ with $\deg_x(\alpha)=1$
and leading term $-x$, and $mb$ roots $\beta$ of $P$ with $\deg_x(\beta)\le 0$.
\end{proposition}
\begin{proof}
Take $\tau_0\coloneqq\pi x^0$. Then $j_0=0$, $\dir(j_0)=(1,0)$ and $\varphi_{\tau_0}=\ide$. By the first equality
in~\eqref{esquina superior uno cero}, we have
 $$
 \en_{\rho,\sigma}(\varphi_{\tau_0}(P))=\en_{1,0}(P)=m(a,b),
 $$
 and by Proposition~\ref{proposition definition pi root}, we have $|D_{\tau_0}^P|=mb$.
Since $\deg_x(\beta)\le j_0=0$ for all $\beta\in D_{\tau_0}^P$,
this yields $mb$ roots with
$\deg_x(\beta)\le 0$. On the other hand, by Proposition~\ref{factor genera pi raiz} with $\tau=\pi x$, $\lambda=-1$ and $\varphi_1(y)=y-x$,  there exists $j_1<1$ such that the
 $\pi$-root $\tau_1\coloneqq -x+\pi x^{j_1}$ satisfies $|D_{\tau_1}^P|=ma$, since $f_{P,\tau}(z)=(z+1)^{ma}z^{mb}$, and so the multiplicity of $\lambda=-1$ is $ma$.
 Moreover, by~\eqref{direccion en el intervalo} and the first equality in~\eqref{esquina superior},
 $$
 \en_{\rho_1,\sigma_1}(\varphi_1(P))=\st_{1,1}(\varphi_1(P))=m(b,a),
 $$
and then $v_{1,-1}(\en_{\rho_1,\sigma_1}(\varphi_1(P)))>0$. So,
 every root $\alpha\in D_{\tau_1}^P$ is a minor root.
\end{proof}

\begin{definition}
Following~\cite{X}, the minor roots in Proposition~\ref{top roots} are called \emph{top minor roots}.
\end{definition}

\begin{proposition}
Let $\alpha\in \mathcal{R}(P)$ be a major root, let $\tau$ be the associated (major) final $\pi$-root and let
$(\rho,\sigma)\coloneqq \dir(\delta_\alpha)$. Then $\left(\frac 1m \en_{\rho,\sigma}(\varphi_\tau(P)),(\rho,\sigma)\right)$ is a regular
corner of type I of $(\varphi_\tau(P),\varphi_\tau(Q))$ (see~\cite{GGV1}*{Definition~5.5} and the discussion
above~\cite{GGV1}*{Remark~5.9}).
\end{proposition}

\begin{proof}
Item~(3) of~\cite{GGV1}*{Definition~5.5} holds by hypothesis, item~(1) holds by the very definition of $\pi$-root, Proposition~\ref{raices menores}
and~\cite{GGV1}*{Theorem 2.6(4)}, and item~(2) holds by Remark~\ref{rho sigma esta en dir de P}. Moreover,
Proposition~\ref{raices menores y raices mayores}(2)a) proves that $\left(\frac 1m \en_{\rho,\sigma}(\varphi_\tau(P)),(\rho,\sigma)\right)$ is of type I.
\end{proof}

\begin{proposition}\label{traduccion}
Let $j_0<j_1<\dots < j_k\in\frac 1l \mathds{Z}$ and let $(\rho,\sigma)\coloneqq\dir(j_0)$.
Consider the automorphism $\varphi$ of $L^{(l)}$ defined by
$$
\varphi(x^{1/l})\coloneqq  x^{1/l}\quad\text{and}\quad \varphi(y)\coloneqq y+\sum_{i=1}^k a_i x^{j_i}.
$$
Let $A=((a/l,b),(\rho,\sigma))$ be a regular corner of
$(\varphi(P),\varphi(Q))$. The following facts hold:
\begin{enumerate}
\item $\tau\coloneqq \sum_{i=1}^k a_ix^{j_i}+\pi x^{j_0}$ is a $\pi$-root of $P$ and of $Q$.
          \item If $A$ is of type Ib, then $\tau$ is a final major $\pi$-root of $P$ and $Q$,
\begin{equation}\label{formulas traducidas}
|D_\tau^{P}|=mb\quad\text{and}\quad\quad |D_\tau^{Q}|=nb.
\end{equation}
Moreover, if $\st_{\rho,\sigma}(\varphi(Q))=(k/l,0)$ for some $1\le k<l-a/b$, then $\lambda_\tau^{Q}=\frac kl$.
\end{enumerate}
\end{proposition}

\begin{proof} (1)\enspace By items~(1) and~(3) of~\cite{GGV1}*{Definition~5.5}, we know that $A=\frac 1m \en_{\rho,\sigma}(\varphi(P))$ and
that $b\ge 1$. Hence, by equalities~\eqref{grado de f es numero de elementos} and~\eqref{igualdad para f de P y tau},
we conclude that $\deg(f_{P,\tau})>0$ and so $\tau$ is a $\pi$-root of $P$. Since by~\cite{GGV1}*{Corollary~5.7} and Remark~\ref{nm par}
the equality $A=\frac 1n \en_{\rho,\sigma}(\varphi(Q))$ holds, and $(Q,P)$ is an $(n,m)$-pair, $\tau$ is also a $\pi$-root of $Q$.

\smallskip

\noindent (2)\enspace The two expressions for $A$ obtained in the proof of item~(1), combined with the equality~\eqref{grado de f es numero de elementos} and the corresponding equality for $Q$, yield
the equalities in~\eqref{formulas traducidas}. Since $A$ is of type Ib,
$$
[\ell_{\rho,\sigma}(\varphi(P)),\ell_{\rho,\sigma}(\varphi(Q))]\ne 0,
$$
and so, by Proposition~\ref{multiple roots}(1), the polynomial $f_{P,\tau}$ has no multiple roots. Moreover, using again
equality~\eqref{grado de f es numero de elementos} and equality~\eqref{formulas traducidas} we obtain that $\deg(f_{P,\tau})=mb>1$.
This proves that $\tau$ is a major final $\pi$-root of $P$, and then, by Proposition~\ref{raices menores y raices mayores}(2)b), also of $Q$.
Finally, assuming that $\st_{\rho,\sigma}(\varphi(Q))=(k/l,0)$, equality~\eqref{igualdad para f de P y tau} for $Q$ implies that $\rho\lambda_\tau^{Q}=v_{\rho,\sigma}(\varphi(Q))=\rho\frac kl$, from which the last assertion follows, since $\rho\ne 0$.
\end{proof}

\begin{example}\label{ejemplo 16}
Consider the family $F_1$ of~\cite{GGHV}, corresponding to an $(m,n)$-pair $(P_0,Q_0)$ as in~\cite{GGV1}*{Corollary~5.21}:
\begin{equation}\label{datos de 16}
A_0=(4,12),\quad A_0'= (1,0),\quad A_1= (7/4,3),\quad k=1,\quad m= 2j+3\quad\text{and}\quad n= 3j+4.
\end{equation}
Then $(P_0,Q_0)$ has the shape given in Figure~\ref{figura 1}.
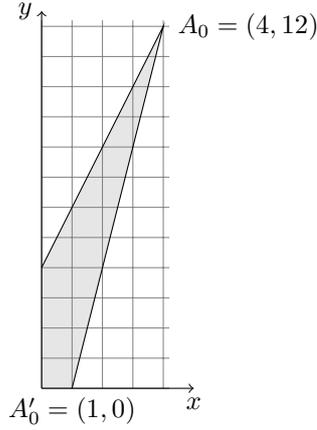
\begin{figure}[ht]
\centering
\begin{tikzpicture}[scale=0.4]
\fill[gray!20] (0,0) -- (1,0) -- (4,12)--(0,4);
\draw[step=1cm,gray,very thin] (0,0) grid (4.2,12.2);
\draw [->] (0,0) -- (5,0) node[anchor=north]{$x$};
\draw [->] (0,0) --  (0,12.5) node[anchor=east]{$y$};
\draw (1,0) node[anchor=north]{$A_0'=(1,0)$} --  (4,12) node[fill=white,right=2pt]{$A_0=(4,12)$} -- (0,4);
\end{tikzpicture}
\caption{Shape of $(P_0,Q_0)$}
\label{figura 1}
\end{figure}
In fact, by~\eqref{datos de 16}, the edge from $A_0$ to $A_0'$ is determined. So we only must prove that
$$
(\rho,\sigma)\coloneqq \Succ_{P_0}(1,0)=\Succ_{Q_0}(1,0)=(-2,1)
$$
 and that $\frac 1m\en_{-2,1}(P)=(0,4)$. By~\cite{GGV1}*{Corollary~5.21(4)} we know that $(-1,1)<(\rho,\sigma)<(-1,0)$. Moreover, by the second equality in~\cite{GGHV}*{(2.13)} we have
$$
q_0=\frac{v_{4,-1}(4,12)}{\gcd(v_{4,-1}(4,12),4-1)}=\frac{4}{\gcd(4,3)}=4.
$$
On the other hand, at the beginning of~\cite{GGV2}*{Subsection 2.4} we see that
$$
\en_{\rho,\sigma}(F_0) = \frac{p_0}{q_0}\frac{1}{m}\en_{\rho_0,\sigma_0}(P_0),
$$
and therefore, by~\cite{GGV1}*{Corollary~7.2}, there exists a $(\rho,\sigma)$-homogeneous element $R$ such that $\ell_{\rho,\sigma}(P)=R^{4m}$. This is only
possible if $(\rho,\sigma)=(-k,1)$ for some $k\in\mathds{N}$, with $k\ge 2$. But $k\ge 3$ leads to $v_{\rho,\sigma}(P_0)\le 0$ and then
$\deg_y(P_0(0,y))\le 0$,  which contradicts~\cite{vdE}*{Proposition~10.2.6}. So $k=2$ and hence
$$
\frac 1m\en_{-2,1}(P)=4\st_{-2,1}(R)=(0,4),
$$
as desired. Since $P\coloneqq \psi(P_0)$ and $Q\coloneqq
\psi(Q_0)$, where $\psi(y)=y$ and $\psi(x)=x+y$ (see the beginning of Subsection~\ref{subseccion 21}), the shape of $P$ is as in Figure~\ref{figura 2}, and $P$ is a monic polynomial in $y$ of degree $16m$.
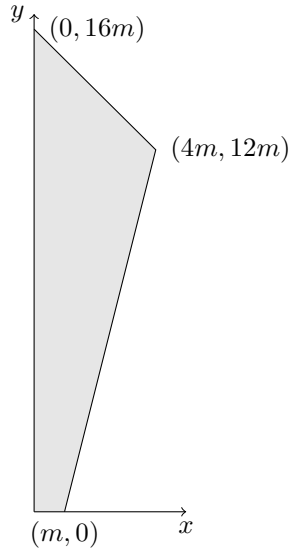
\begin{figure}[ht]
\centering
\begin{tikzpicture}[scale=0.4]
\fill[gray!20] (0,0) -- (1,0) -- (4,12)--(0,16);
%\draw[step=1cm,gray,very thin] (0,0) grid (4.2,12.2);
%
\draw [->] (0,0) -- (5,0) node[anchor=north]{$x$};
\draw [->] (0,0) --  (0,16.5) node[anchor=east]{$y$};
\draw (1,0) node[anchor=north]{$(m,0)$} --  (4,12) node[fill=white,right=2pt]{$(4m,12m)$} -- (0,16)node[right=2pt]{$(0,16m)$};
\end{tikzpicture}
\caption{Shape of $P_\xi$}
\label{figura 2}
\end{figure}
Write $\ell_{4,-1}(P)=x^{m}g(z)^m$, where $z=x^{1/4}y$.
By~\cite{GGV2}*{Theorem~2.20(6)} and the fact that $v_{1,-1}(A_0')>0$, we know that $(A_0,(\rho,\sigma))=((4,12),(4,-1))$ is a regular corner of
type~IIb) of $(P,Q)$. Hence, by item~(8) of the same theorem, $v_{0,1}(A_1)=\frac{m_\lambda}{m}$, where $m_\lambda$ is the multiplicity of
$z-\lambda$ in $\mathfrak{p}_0(z)\coloneqq g(z)^m$.
Since $v_{0,1}(A_1)=3$, by~\cite{GGV2}*{Remarks~3.8 and~3.9} we have
$$
g(z)=\lambda_0 (z^4-\lambda_1^4)^{3},
$$
for some $\lambda_0,\lambda_1\in K^{\times}$. It follows that
$$
\ell_{4,-1}(P_\xi)=\lambda_0 x^{m}(z-\lambda_1)^{3m}(z-i\lambda_1)^{3m}(z+\lambda_1)^{3m}(z+i\lambda_1)^{3m},
$$
and so we have four final major $\pi$-roots
$$
\tau_0 \coloneqq \lambda_1 x^{1/4}+\pi x^{\delta},\quad
\tau_1  \coloneqq i\lambda_1 x^{1/4}+\pi x^{\delta},\quad
\tau_2  \coloneqq -\lambda_1 x^{1/4}+\pi x^{\delta} \quad\text{and}\quad
\tau_3  \coloneqq -i\lambda_1 x^{1/4}+\pi x^{\delta},
$$
where $\delta=\sigma/\rho$, with $(\rho,\sigma)\coloneqq \dir\left(m\left(\frac 74,3\right)-\left(\frac 34,1\right)\right)$. Here
$A_1=\left(\frac 74,3\right)$ is the same final corner (see~\cite{GGV2}*{Definition~2.18}) for all major final roots, corresponding to the regular
corner $(A_1,(\rho,\sigma))$  of type Ib) of each of the four
$(m,n)$-pairs $(\varphi_{\tau_j}(P),\varphi_{\tau_j}(Q))$. By the first equality in~\eqref{formulas traducidas}, there are $3m$ roots of $P$ associated to each of these major roots,
and by Proposition~\ref{top roots}, the remaining $4m$ roots of $P$ are minor roots. Now we compute
$$
I_M=\sum_{\tau\in P_M}|D_\tau^{P}|\lambda_\tau^Q=\sum_{j=0}^3|D_{\tau_j}^{P}|\lambda_{\tau_j}^Q=4\cdot mb\cdot \frac{k}{l} =4\cdot
m\cdot 3\cdot \frac 14=3m=3(2j+3).
$$
\end{example}

\newpage
\subsection{Intersection number and minor roots}
For the sake of brevity in the sequel we write $P_x$, $Q_x$, $P_y$ and $Q_y$ instead of the partial derivatives $\partial_x P$,
$\partial_x Q$, $\partial_y P$ and $\partial_y Q$, respectively.

\begin{lemma}
\label{lema auxiliar}
 Let $(P,Q)$ be as above, $(\rho,\sigma)$ be a direction with $\rho\ne 0$ and $\alpha\in\mathcal{R}(P)$.
Write $\ell_{\rho,\sigma}(P)= x^{u}g(z)$ with $z\coloneqq x^{-\sigma/\rho}y$.
The following facts hold:
\begin{enumerate}
  \item If $\deg(g)>0$, then $\ell_{\rho,\sigma}(P_y)= x^{u-\sigma/\rho}g'(z)$.
  \item $\alpha$ is a minor root if and only if $\deg_x(Q(\alpha))=0$.
  \item Let $\beta\in\mathcal{R}(P_y)$. There exists $\tau\in P_m$ such that $\beta\in D_\tau^{P_y}$ if and only if $\deg_x(P(\beta))=0$.
  \item If $\alpha$ is a minor root, then $\deg_x(P_y(\alpha))=-\delta_\alpha$.
  \item Let $\tau\in P_m$ and assume that $f_{P_y,\tau}$ and $f_{Q_y,\tau}$ are coprime. Then $\deg_x(Q_y(\beta))=-\delta_\tau$ for all
  $\beta\in D_\tau^{P_y}$.
  \item Let $\tau\in P_m$ and assume that $f_{P_y,\tau}$ and $f_{Q_y,\tau}$ are not coprime. Then there exists $\beta\in D_\tau^{P_y}$, such that $\deg_x(Q_y(\beta))<-\delta_\tau$.
\end{enumerate}
\end{lemma}

\begin{proof}
(1)\enspace This follows from the fact that the morphism $\partial_y$ satisfies $\partial_y(x^iy^j)=jx^iy^{j-1}$ for $j>0$, and so
$$
v_{\rho,\sigma}(\partial_y(x^iy^j))=v_{\rho,\sigma}(x^i y^j)-\sigma.
$$
Hence $\ell_{\rho,\sigma}(\partial_y P)=\partial_y \ell_{\rho,\sigma}(P)$ when $\partial_y \ell_{\rho,\sigma}(P)\ne 0$, and so
$$
\ell_{\rho,\sigma}(P_y)=\partial_y(x^{u}g(z))=x^{u-\sigma/\rho}g'(z),
$$
because $\deg(g)>0$.

\smallskip

\noindent (2)\enspace By items~(1)a) and~(2)c) of Proposition~\ref{raices menores y raices mayores}, we know that
 $\alpha$ is a minor root if and only if $\lambda_\tau^Q=0$ for the $\pi$-root $\tau$ associated to $\alpha$.
This proves~(2), since $\lambda_\tau^Q=\deg_x(Q(\alpha))$ by Lemma~\ref{Q en alpha es Q en tau}.

\smallskip

\noindent (3)\enspace Define
$$
\delta_\beta\coloneqq \min\{\deg_x(\alpha - \beta) | \alpha\in\mathcal{R}(P)\}.
$$
Write
$$
\beta=\sum_{j>\delta_\beta} a_j x^j+\lambda x^{\delta_\beta}+\sum_{j<\delta_\beta} a_j x^j
$$
Then $\tau\coloneqq \sum_{j>\delta_\beta} a_j x^j+\pi x^{\delta_\beta}$ is a $\pi$-root of $P$. Since $0<|D_\tau^{P_y}|= |D_\tau^{P}|-1$,
by Remark~\ref{raiz con lambda negativo} we know that $\lambda_\tau^P\ge 0$.  Take $\alpha\in D_\tau^{P}$ and let $\tau_1$ be the final $\pi$-root of $P$
associated with $\alpha$. We have
$\delta_\alpha\le \delta_\beta$ (since $\delta_\beta<\delta_\alpha$ implies $|D_{\tau}^P|=1$), hence $\lambda_{\tau_1}\le \lambda_{\tau}$ and so $\lambda_\tau^P= 0$ if and
only if $\tau=\tau_1$ is a final minor $\pi$-root of $P$.

We claim that $\lambda_\tau^P=\deg_x(P(\beta))$.  In fact,
$f_{P,\tau}(\lambda)\ne 0$ since otherwise, by Proposition~\ref{factor genera pi raiz}, there exists $j_1<\delta_\beta$ such that
the $\pi$-approximation of $\beta$ up to $j_1$ is a $\pi$-root of $P$, contradicting the minimality of $\delta_\beta$.
Hence, by Proposition~\ref{P en tau coincide con P en alpha} we have $\deg_x(P(\beta))=\lambda_\tau^P\ge 0$. Hence, if $\deg_x(P(\beta))=0$,
then $\beta\in D_\tau^{P_y}$ and  $\tau\in P_m$.
On the other hand, if $\beta\in D_{\tau_2}^{P_y}$ for some  $\tau_2\in P_m$,  then $\delta_\beta\le \delta_{\tau_2}$, hence
$0\le \lambda_{\tau}\le \lambda_{\tau_2}=0$, and so $0=\lambda_\tau^P=\deg_x(P(\beta))$, as desired.

\smallskip

\noindent (4)\enspace Let $\tau\coloneqq \sum_{j>\delta_\alpha} a_j x^j+\pi x^{\delta_\alpha}$ be the minor final $\pi$-root of $P$ associated with $\alpha$.
Write
$$
\alpha=\sum_{j>\delta_\alpha} a_j x^j+\lambda x^{\delta_\alpha}+\sum_{j<\delta_\alpha} a_j x^j
$$
Since $f_{P,\tau}(\lambda)=0$ and $f_{P,\tau}$ has no multiple roots, we have $f'_{P,\tau}(\lambda)\ne 0$.
But by item~(1) we have $f_{P_y,\tau}=f'_{P,\tau}$, and so, by Proposition~\ref{P en tau coincide con P en alpha},
we obtain $\lambda_\tau^{P_y}=\deg_x(P_y(\tau))=\deg_x(P_y(\alpha))$. Using again item~(1) we have
$\lambda_\tau^{P_y}=\lambda_\tau-\sigma/\rho$, and since $\lambda_\tau=0$, the result follows immediately.

\smallskip

\noindent (5)\enspace Write $\tau\coloneqq \sum_{j>j_0} a_j x^j+\pi x^{j_0}$ and let $\beta\in D_{\tau}^{P_y}$.
Write
$$
\beta=\sum_{j>j_0} a_j x^j+\lambda x^{j_0}+\sum_{j<j_0} a_j x^j.
$$
Since by Corollary~\ref{lambda tiene que ser raiz} we know that $f_{P_y,\tau}(\lambda)=0$, and $f_{P_y,\tau}$ is coprime with $f_{Q_y,\tau}$, we have $f_{Q_y,\tau}(\lambda)\ne 0$.
Hence, by Proposition~\ref{P en tau coincide con P en alpha},
we obtain $\lambda_\tau^{Q_y}=\deg_x(Q_y(\tau))=\deg_x(Q_y(\beta))$ and by item~(1) we have
$\lambda_\tau^{Q_y}=\lambda^{Q}_\tau-\sigma/\rho$, and since $\lambda^Q_\tau=0$, the result follows immediately.

\smallskip

\noindent (6)\enspace Write $\tau\coloneqq \sum_{j>j_0} a_j x^j+\pi x^{j_0}$.
Let $\lambda\in K$ such that $f_{Q_y,\tau}(\lambda)=0=f_{P_y,\tau}(\lambda)$. By Proposition~\ref{factor genera pi raiz}
there exist $j_1,j_2<j_0$ such that
$\tau_1\coloneqq \sum_{j>j_0} a_j x^j+\lambda x^{j_0}+\pi x^{j_1}$  is a $\pi$-root of $P_y$ and
$\tau_2\coloneqq \sum_{j>j_0} a_j x^j+\lambda x^{j_0}+\pi x^{j_2}$  is a $\pi$-root of $Q_y$. Take $j_3\coloneqq \max\{j_1,j_2\}$ and
so $\tau_3\coloneqq \sum_{j>j_0} a_j x^j+\lambda x^{j_0}+\pi x^{j_3}$  is a $\pi$-root of $Q_y$ and $P_y$.
Take $\beta\in D_{\tau_3}^{P_y}$. Then
$$
\beta=\sum_{j>j_0} a_j x^j+\lambda x^{j_0}+\sum_{j<j_0} a_j x^j,
$$
and set $T\coloneqq \lambda x^{j_0}+\sum_{j<j_0} a_j x^j$. Then
$$
Q_y(\beta)=\ev_{y=T}(\varphi_\tau(Q_y))=\ev_{y=\lambda x^{j_0}}(\ell_{\rho,\sigma}(\varphi_\tau(Q_y)))+R=x^{\lambda_\tau^{Q_y}}f_{Q_y,\tau}(\lambda)+R
$$
for some $R$ with $v_{\rho,\sigma}(R)<v_{\rho,\sigma}(\varphi_\tau(Q_y)))=\rho \lambda_\tau^{Q_y}$. Since $f_{Q_y,\tau}(\lambda)=0$, we obtain
$$
\rho \deg_x(Q_y(\beta))=v_{\rho,\sigma}(Q_y(\beta))<\rho \lambda_\tau^{Q_y}.
$$
Since by item~(1) we know that
$\lambda_\tau^{Q_y}=\lambda^{Q}_\tau-\sigma/\rho$, and since $\lambda^Q_\tau=0$ we have
$$
\deg_x(Q_y(\beta))<-\sigma/\rho,
$$
as desired.
\end{proof}

\begin{lemma} \label{derivada del jacobiano}
For any $\alpha\in K((x^{-1/l}))$ we have
\begin{equation*}
  Q_y(\alpha)\frac d{dx}P(\alpha)-P_y(\alpha)\frac{d}{dx}Q(\alpha)\in K^{\times}
\end{equation*}
\end{lemma}

\begin{proof}

\end{proof}

\begin{theorem}\label{numero de interseccion con raices menores}
Set $I_m=1-\sum_{\tau\in P_m}(\delta_{\tau}+1)$. Then $I_m\le I(P,Q)$. We also have
\begin{equation}\label{Inter 1}
I(P,P_yQ)= \deg(P)-\sum_{\tau\in P_m}| D_\tau^P|(1+\delta_\tau).
\end{equation}
\end{theorem}

\begin{proof}
It suffices to prove~\eqref{Inter 1}
and
\begin{equation}\label{Inter 2}
  I(P,P_y)\le \deg(P)-1-\sum_{\tau\in P_m}(|D_\tau^P|-1)(\delta_\tau+1).
\end{equation}
In fact, equalities~\eqref{Inter 1} and~\eqref{Inter 2} yield
$$
I(P,Q)=I(P,P_yQ)-I(P,P_y)=1-\sum_{\tau\in P_m}(\delta_{\tau}+1),
$$
as desired.

\smallskip

\noindent Proof of equality~\eqref{Inter 1}.\enspace
By Lemma~\ref{derivada del jacobiano}, for each $\alpha\in\mathcal{R}(P)$ we have $P_y(\alpha)\frac{d}{dx}Q(\alpha)\in K^{\times}$. Moreover,
by Lemma~\ref{lema auxiliar}(2), if $\alpha$ is a major root, then $\deg_x(P_y(\alpha)Q(\alpha))=1$. On the other hand,
if $\alpha$ is a minor root, then by Proposition~\ref{raices menores y raices mayores}(1)a),
Lemma~\ref{Q en alpha es Q en tau} and Lemma~\ref{lema auxiliar}(4), we have
$$
\deg_x(P_y(\alpha)Q(\alpha))=\deg_x(P_y(\alpha))=-\delta_\alpha=-\delta_\tau,
$$
where $\tau$ is the minor final $\pi$-root associated with $\alpha$.
Using this facts we obtain
\begin{equation*}\label{interseccion de P y P sub y y Q}
\begin{aligned}
I(P,P_yQ) & =\sum_{\alpha\in\mathcal{R}(P)} \deg_x(P_y(\alpha)Q(\alpha))\\
& = \sum_{\tau\in P_m}\sum_{\alpha\in D_\tau^P} \deg_x(P_y(\alpha)Q(\alpha))+ \sum_{\tau\in P_M}\sum_{\alpha\in D_\tau^P} \deg_x(P_y(\alpha)Q(\alpha))\\
& = \sum_{\tau\in P_m}| D_\tau^P|(-\delta_\tau)+\sum_{\tau\in P_M}| D_\tau^P|+\sum_{\tau\in P_m}| D_\tau^P|-\sum_{\tau\in P_m}| D_\tau^P|\\
&= \deg(P)-\sum_{\tau\in P_m}| D_\tau^P|(1+\delta_\tau),
\end{aligned}
\end{equation*}
where the first equality is obtained as in the proof of Theorem~\ref{Interseccion con raices mayores}.

\smallskip

\noindent Proof of inequality~\eqref{Inter 2}.\enspace
By Lemma~\ref{derivada del jacobiano}, for each $\beta\in\mathcal{R}(P_y)$, we have $Q_y(\beta)\frac{d}{dx}P(\beta)\in K^{\times}$.
Define
$$
P_{y,m}\coloneqq \{\beta\in\mathcal{R}(P_y) : \text{ there exists a minor final $\pi$-root $\tau$ of $P$, such that $\beta\in D_\tau^{P_y}$}\}.
$$
Then,
by Lemma~\ref{lema auxiliar}(3), if $\beta$ is not in $P_{y,m}$, then $\deg_x(Q_y(\beta)P(\beta))=1$.
On the other hand, if $\beta$ is in $P_{y,m}$,
then by items~(3), (5) and~(6) of Lemma~\ref{lema auxiliar}, we have
$$
\deg_x(P(\beta)Q_y(\beta))=\deg_x(Q_y(\beta))\le-\delta_\tau,
$$
where $\tau$ is the minor final $\pi$-root associated with $\beta$.
Using this facts we obtain
\begin{equation}\label{interseccion de P y P y Q sub y}
\begin{aligned}
I(P_y,PQ_y) & =\sum_{\beta\in\mathcal{R}(P_y)} \deg_x(P(\beta)Q_y(\beta))\\
& = \sum_{\tau\in P_m}\sum_{\beta\in D_\tau^{P_y}} \deg_x(P(\beta)Q_y(\beta))+ \sum_{\beta\notin P_{y,m}} \deg_x(P(\beta)Q_y(\beta))\\
& \le \sum_{\tau\in P_m}| D_\tau^{P_y}|(-\delta_\tau)+\deg(P_y)-\sum_{\tau\in P_m}| D_\tau^{P_y}|\\
&= \deg(P)-1-\sum_{\tau\in P_m}| D_\tau^{P_y}|(1+\delta_\tau)\\
&= \deg(P)-1-\sum_{\tau\in P_m}(| D_\tau^{P}|-1)(1+\delta_\tau).
\end{aligned}
\end{equation}
Since, by the Jacobian condition,
$$
\Res_y(P_y,Q_y )\Res_y(P_y, P_x)=\Res_y(P_y,Q_y P_x) = \prod_{\beta\in\mathcal{R}(P_y)}Q_y(\beta)P_x(\beta)=1,
$$
we have $I(P_y,Q_y)=0$, and
so~\eqref{interseccion de P y P y Q sub y} yields inequality~\eqref{Inter 2}.
\end{proof}

\end{document}